\documentclass[a4paper,11pt,leqno]{article}
\usepackage{bbm}
\usepackage{amssymb, amsmath, amscd, amsfonts, amsthm}
\usepackage[body={13cm, 21cm}, centering, dvipdfm]{geometry}
\usepackage[tiny]{titlesec}
\usepackage[symbol]{footmisc}
\usepackage{mathrsfs}
\usepackage{enumerate, varioref}
\usepackage{graphicx}
\usepackage[square, numbers, sort&compress]{natbib}
\usepackage{marvosym}
\usepackage[colorlinks,linkcolor=red,anchorcolor=blue,citecolor=green,hyperindex,dvipdfm]{hyperref}
\DefineFNsymbols{nu}{\dag 1 2}\setfnsymbol{nu}
\title{\textbf{Optimal stopping problems with regime switching: a viscosity solution method}\footnotemark{}}

\author{Yong-Chao Zhang\footnotemark{}\; and Na Zhang\footnotemark{}}

\date{}
 \newtheorem{thm}{Theorem}[section]
 \newtheorem{cor}[thm]{Corollary}
 \newtheorem{lem}[thm]{Lemma}
 \theoremstyle{remark}
 \newtheorem{rem}[thm]{Remark}
 \newtheorem{exm}[thm]{Example}
 \newtheorem{defn}[thm]{Definition}

\begin{document}
\footnotetext[1]{This work is partially supported by the Fundamental Research Funds for the Central Universities grant N142303010.}
\footnotetext[2]{School of Mathematics and Statistics, Northeastern University at Qinhuangdao, Taishan Road
143, Qinhuangdao 066004, China. E-mail: ldfwq@163.com.}
\footnotetext[3]{\Letter\; School of Mathematical Sciences, Nankai University, Weijin Road 94, Tianjin 300071, China. E-mail: nazhang0804@163.com.  Tel: +86 186 3356 0435.}
\maketitle
\vspace{-0.5cm}
\begin{abstract}
We employ the viscosity solution technique to analyze optimal stopping problems with regime switching. Specifically, we obtain the viscosity property of value functions, the uniqueness of viscosity solutions, the regularity of value functions and the form of optimal stopping intervals. Finally, we provide an application of the results.
\end{abstract}

\section{Introduction}\label{intro}
Regime-switching processes are appropriate candidates for describing the price of financial assets \cite{CherGallGhysTauc2003, BaKiMu2014} and the price of some commodities \cite{CasaDufrRout2005}. In addition, as Elias \textit{et al.}~\cite{ElWaFa2014} point out, regime-switching processes are also plausible choices of modelling the stochastic behavior of temperature.  Last but not the least, regime-switching processes appear in real option pricing \cite{Bollen1999}. Thus it is reasonable to consider the optimal stopping problems in which underlying processes and payoff functions are modulated by Markov chains.

Many specific problems of optimal stopping with regime switching have been studied. Assuming that the stock price follows a geometric
Brownian motion modulated by a two-state Markov chain, Guo \cite{Guo2001} provides an explicit closed solution for Russian options. Under the same
assumption as that of \cite{Guo2001}, Guo and Zhang derive an explicit closed solution for perpetual American options in \cite{GuoZhang2004} and
for optimal selling rules in \cite{GuoZhang2005}, respectively, and Buffington and Elliot \cite{BuffElli2002a} explore American options with finite maturity date. Eloe \textit{et al.} \cite{EloeLiuYinZhang2008} develop optimal selling rules via using a regime-switching exponential Gaussian diffusion model. In a regime-switching L\'evy model, Boyarchenko and Levendorski\v{i} \cite{BoyaLeve2008} show a pricing procedure for perpetual American and real options which is efficient even though the number of states is large provided transition rates are not large with respect to riskless rates. D'Auriaa and Kellab study in \cite{d2012markov} Markov modulation of a two-sided reflected Brownian motion and give an application to fluid queues.

The method used in \cite[e.g.,][]{Guo2001, GuoZhang2004, GuoZhang2005} is to construct a solution to some equations by guessing a priori a strategy and then validate it by a verification argument. In this paper, we will employ the viscosity solution technique to determine the solution of optimal stopping problems with regime switching. First, we prove the value function is a viscosity solution of some variational inequalities. Second, we prove the uniqueness of viscosity solutions. Third, we show the regularity of the value function. Finally, we determine the form of optimal stopping intervals.

We outline the structure of this paper. In Section \ref{optstop}, we prove the viscosity property of the value function of optimal stopping problems with regime switching (Theorem \ref{t:VisSol}), the uniqueness of viscosity solutions (Theorem \ref{t:Uniq1}), the regularity of the value function (Theorem \ref{t:Regu}) and the form of optimal stopping intervals (Theorem \ref{t:OptiStraForm}). In Section \ref{model}, we provide an application of the results obtained in Section \ref{optstop}. Some conclusions are drawn in Section \ref{Conclu}.

\section{Optimal Stopping Problems with Regime Switching}\label{optstop}
Let $(\Omega,$ $\mathscr{F},\{\mathscr{F}_t\}_{t\geq 0},\mathbb{P})$ be a filtered probability space with the filtration $\{\mathscr{F}_t\}_{t\geq 0}$ satisfying the usual conditions and $\mathscr{F}_0$ being the completion of $\{\emptyset, \Omega\}$.

Let $X:=(X(t), t\geq 0)$ be a time homogeneous Markov chain defined on $(\Omega, \mathscr{F}, \{\mathscr{F}_t\}_{t\geq
0}, \mathbb{P})$ taking values in the standard orthogonal basis of $\mathbb{R}^m$,  $I:=\{e_1, e_2, \dots, e_m\}$, whose rate matrix is
$A:=(a_{ij})_{m\times m}$ with $a_{ii}<0$, for each $i=1,2,\cdots, m$. Then as in \cite{BuffElli2002}, we can show that
\begin{equation*}
X(t)=X_0+\int_0^t A X(s)\mathrm{d}s+M(t),
\end{equation*}
where $M:=(M(t), t\geq 0)$ is a martingale with respect to the filtration generate by $X$.

Let $B:=(B(t), t\geq 0)$ be a one dimensional standard Brownian motion, which is independent of $X$, defined on $(\Omega, \mathscr{F}, \{\mathscr{F}_t\}_{t\geq
0}, \mathbb{P})$.

Assume that the process $Y:=(Y(t), t\geq 0)$ satisfies
\[
\mathrm{d}Y(t)=\alpha(X(t),Y(t)) \mathrm{d}t+\beta(X(t),Y(t)) \mathrm{d}B(t), \;\; X(0)=e_i, \;\;Y(0)=y,
\]
where $\alpha: I\times\mathbb{R}\rightarrow\mathbb{R}$ and $\beta: I\times\mathbb{R}\rightarrow\mathbb{R}$ are two functions such that $\alpha(e_i,\cdot)$ and
$\beta(e_i,\cdot)$ are Lipschitz continuous for each $i=1,2,\cdots, m$. We assume that $\beta(\cdot,\cdot)>0$.

Let $\mathcal{T}$ denote the set of all stopping times. For any $\tau\in\mathcal{T}$, $e_i\in I$ and $y\in \mathbb{R}$, we define
\[
\begin{aligned}
J^{\tau}(e_i, y):=\mathbb{E}&\left[\left.\int_{0}^{\tau}\exp(-rt)f(X(t),Y(t))\mathrm{d}t\right.\right.\\
&\qquad\left.\left.+\exp(-r\tau)g(X(\tau),Y(\tau))\right|X(0)=e_i,Y(0)=y\right],
\end{aligned}
\]
where $r$ is a real number, and $f: I\times\mathbb{R}\rightarrow\mathbb{R}$ and $g: I\times\mathbb{R}\rightarrow\mathbb{R}$ are two functions such that
$f(e_i,\cdot)$ and $g(e_i,\cdot)$ are Lipschitz continuous for each $i=1,2,\cdots, m$. We assume that $\exp(-r\tau)g(X(\tau),Y(\tau))=0$ on
$\{\tau=\infty\}$.

Then the optimal stopping problem with regime switching is described as follows.
\begin{equation}\label{e:genopt}
\text{Find $V(e_i,y)$ and $\tau^*\in\mathcal{T}$ such that
$V(e_i,y)=\sup\limits_{\tau\in\mathcal{T}}J^{\tau}(e_i,y)=J^{\tau^*}(e_i,y)$.}
\end{equation}

For each $i=1, 2, \cdots, m$, setting $\alpha_i(\cdot):=\alpha(e_i,\cdot)$, $\beta_i(\cdot):=\beta(e_i,\cdot)$, $f_i(\cdot):=f(e_i,\cdot)$,
$g_i(\cdot):=g(e_i,\cdot)$ and $V_i(\cdot):=V(e_i,\cdot)$, we have the following theorem.

\begin{thm}\label{t:VisSol}
Assume that $r$ is large enough. Then for each fixed $i\in\{1,2,\cdots,m\}$, $V_i$ is the unique viscosity solution with at most linear growth of the following variational inequality,
\begin{equation}\label{e:VariIneq}
\min\{rV_i-\mathcal{L}_iV_i-\sum\limits_{q=1}^{m}a_{qi}V_q-f_i, V_i-g_i\}=0\;\,\mathrm{on}\;\,\mathbb{R},
\end{equation}
where $\mathcal{L}$ is defined by
\begin{equation*}
\mathcal{L}_i\xi(y):=\frac{1}{2}\beta_i(y)^2\frac{\partial^2\xi(y)}{\partial y^2}
+\alpha_i(y)\frac{\partial\xi(y)}{\partial y},
\end{equation*}
for any $\xi\in C^2(\mathbb{R})$.
\end{thm}
\begin{rem}
The assumption that {\em $r$ is large enough} is nothing but to guarantee the Lipschitz continuity of the functions $V_i$'s. Refer to Lemma \ref{l:LipsCont}.
\end{rem}
Before proving the above theorem, let us first recall the definition of viscosity solutions.

The theory of viscosity solutions applies to certain partial differential equations of the form $F(x, u, Du, D^2u)=0$ where $F:
\mathbb{R}^N\times\mathbb{R}\times\mathbb{R}^N\times\mathcal{S}(N)\rightarrow\mathbb{R}$ along with $\mathcal{S}(N)$ is the set of all symmetric $N\times N$
matrices.

We require $F$ to satisfy the monotonicity condition
\begin{equation}\label{e:con}
F(x, r, p, X)\leq F(x, s, p, Y),\,\,\mathrm{whenever}\,\, r\leq s\,\, \mathrm{and}\,\, Y\leq X.
\end{equation}
Here $r, s\in \mathbb{R}$, $x, p\in\mathbb{R}^N$, $X, Y\in\mathcal{S}(N)$ and $\mathcal{S}(N)$ is equipped with its usual order.

It will be convenient to have the following notations.
\begin{equation*}
\begin{split}
&\mathrm{USC}(\mathcal{O}):=\{\textrm{upper semicontinuous functions}\,\, u:\mathcal{O}\rightarrow\mathbb{R}\},\\
&\mathrm{LSC}(\mathcal{O}):=\{\textrm{lower semicontinuous functions}\,\, u:\mathcal{O}\rightarrow\mathbb{R}\},\\
\end{split}
\end{equation*}
\begin{equation*}
\begin{split}
&J_\mathcal{O}^{2, +}u(x):=\{(D\varphi(x), D^2\varphi(x)): \varphi \,\,\mathrm{is}\,\, C^2\,\, \mathrm{and}\,\, u-\varphi\,\,\textrm{has a maximum at}\,\,
x\},\\
&J_\mathcal{O}^{2, -}u(x):=\{(D\varphi(x), D^2\varphi(x)): \varphi \,\,\mathrm{is}\,\, C^2\,\, \mathrm{and}\,\, u-\varphi\,\,\textrm{has a minimum at}\,\,
x\},\\
\end{split}
\end{equation*}
\begin{equation*}
\begin{split}
&\overline{J}_\mathcal{O}^{2, +}u(x):=\{(p, X)\in\mathbb{R}^N\times\mathcal{S}(N): \textrm{there is a sequence}\,\,(x_n,p_n,X_n)_{n\in\mathbb{N}}\subset\\
&\qquad\qquad\qquad\mathcal{O}\times J_\mathcal{O}^{2, +}u(x_n)\,\,\textrm{such that}\,\,(x_n,u(x_n),p_n,X_n)\,\,\textrm{converges to}\\
&\qquad\qquad\qquad(x,u(x),p,X)\,\,\textrm{as}\,\,n\rightarrow\infty\},\\
&\overline{J}_\mathcal{O}^{2, -}u(x):=\{(p, X)\in\mathbb{R}^N\times\mathcal{S}(N): \textrm{there is a sequence}\,\,(x_n,p_n,X_n)_{n\in\mathbb{N}}\subset\\
&\qquad\qquad\qquad\mathcal{O}\times J_\mathcal{O}^{2, -}u(x_n)\,\,\textrm{such that}\,\,(x_n,u(x_n),p_n,X_n)\,\,\textrm{converges to}\\ &\qquad\qquad\qquad(x,u(x),p,X)\,\,\textrm{as}\,\,n\rightarrow\infty\}.
\end{split}
\end{equation*}

\begin{defn}{\cite[Definition 2.2]{CIL92}}
Let $F$ satisfy (\ref{e:con}) and $\mathcal{O}\subset\mathbb{R}^N$. A viscosity subsolution of $F=0$ (equivalently, a viscosity solution of $F\leq 0$) on
$\mathcal{O}$ is a function $u\in \mathrm{USC}(\mathcal{O})$ such that
\begin{equation*}
F(x, u(x), p, X)\leq 0,\,\,\textrm{for all}\,\, x\in\mathcal{O}\,\, \mathrm{and}\,\, (p, X)\in J_\mathcal{O}^{2, +}u(x).
\end{equation*}

Similarly, a viscosity supersolution of $F=0$ (equivalently, a viscosity solution of $F\geq 0$) on $\mathcal{O}$ is a function $u\in
\mathrm{LSC}(\mathcal{O})$ such that
\begin{equation*}
F(x, u(x), p, X)\geq 0,\,\,\textrm{for all}\,\, x\in\mathcal{O}\,\, \mathrm{and}\,\, (p, X)\in J_\mathcal{O}^{2, -}u(x).
\end{equation*}

Finally, $u$ is a viscosity of $F=0$ in $\mathcal{O}$ if it is both a viscosity subsolution and a viscosity supersolution of $F=0$ in $\mathcal{O}$.
\end{defn}

The following lemma is used in the proof of Theorem \ref{t:VisSol}.
\begin{lem}\label{l:LipsCont}
There is a positive number $r_0$ such that $V_i$ is Lipschitz continuous for each $i=1,2,\cdots,m$ if $r>r_0$.
\end{lem}
\begin{proof}
The proof is similar to that of \cite[p.~96, Lemma 5.2.1]{Pham2009}.
\end{proof}

Now we provide a proof of Theorem \ref{t:VisSol}.
\begin{proof}[Proof of Theorem \ref{t:VisSol}]
1. Following \cite[Section 4.2]{Touzi2013}, we have the dynamic programming principle \eqref{e:DynProPri}. Also refer to \cite[Theorem 4]{el1982optimal}. For any $j\in \{1,2,\cdots,m\}$, $y\in\mathbb{R}$ and $\theta\in\mathcal{T}$, we have
\begin{equation}\label{e:DynProPri}
\begin{aligned}
V_j(y)=\sup\limits_{\tau\in\mathcal{T}}\mathbb{E}&\left[\int_0^{\theta\wedge\tau}\exp(-rt)f(X(t),Y(t))\mathrm{d}t\right.\\
&\left.+\exp(-r\tau)g(X(\tau),Y(\tau))\mathbbm{1}_{\{\theta>\tau\}}\right.\\
&+\left.\exp(-r\theta)V(X(\theta),Y(\theta))\mathbbm{1}_{\{\theta\leq\tau\}}|X(0)=e_j,Y(0)=y\right].
\end{aligned}
\end{equation}
2. \textit{The viscosity supersolution property.} Fix an $i\in \{1,2,\cdots,m\}$. Suppose that $\varphi_i\in C^2(\mathbb{R})$ and $V_i-\varphi_i$ has a
minimum at some point $y_i\in\mathbb{R}$ such that $V_i(y_i)=\varphi_i(y_i)$. For any positive integer $k$, choose some functions $\varphi_j^k\in
C^2(\mathbb{R})$ such that $\max\limits_{y\in\mathbb{R}}|\varphi_j^k(y)-V_j(y)|<1/k$, where $j\neq i$ (see, for example, \cite[p.~63, Lemma 1.3]{Zhang1996}). Set $\varphi_i^k:=\varphi_i$ for convenience.

Define a stopping time $\theta_0:=\inf\{t:t>0,X(t)\neq e_i, |Y(t)-y_i|>1\}$. Then by the dynamic programming principle (\ref{e:DynProPri}), for any positive
number $\varepsilon$, we have, via taking $\tau=\theta=\theta_0\wedge\varepsilon$,
\begin{equation}\label{e:DynProPri1}
\begin{aligned}
&V_i(y_i)\geq\mathbb{E}\left[\left.\int_0^{\theta_0\wedge\varepsilon}\exp(-rt)f(X(t),Y(t))\mathrm{d}t\right.\right.\\
&\qquad\left.\left.+\exp(-r(\theta_0\wedge\varepsilon))V(X(\theta_0\wedge\varepsilon),Y(\theta_0\wedge\varepsilon))\right|X(0)=e_i,Y(0)=y_i\right]\\
&\geq-\frac{1}{k}+\mathbb{E}\left[\left.\int_0^{\theta_0\wedge\varepsilon}\exp(-rt)f(X(t),Y(t))\mathrm{d}t\right.\right.\\
&\quad\left.\left.+\exp(-r(\theta_0\wedge\varepsilon))\sum\limits_{j=1}^{m}\varphi_j^k(Y(\theta_0\wedge\varepsilon))X_j(\theta_0\wedge\varepsilon)
\right|X(0)=e_i,Y(0)=y_i\right].
\end{aligned}
\end{equation}

In addition, by It\^{o}'s formula, we have
\begin{equation}\label{e:DynIto}
\begin{aligned}
&\quad\exp(-r(\theta_0\wedge\varepsilon))\sum\limits_{j=1}^{m}\varphi_j^k(Y(\theta_0\wedge\varepsilon))X_j(\theta_0\wedge\varepsilon)\\
&=\sum\limits_{j=1}^{m}\varphi_j^k(Y(0))X_j(0)+\sum\limits_{j=1}^{m}\int_0^{\theta_0\wedge\varepsilon}
\exp(-rt)\left(\frac{1}{2}\beta(X(t),Y(t))^2\frac{\partial^2\varphi_j^k}{\partial y^2}(Y(t))\right.\\
&\quad\left.+\alpha(X(t),Y(t))\frac{\partial\varphi_j^k}{\partial y}(Y(t))-r\varphi_j^k(Y(t))+\sum\limits_{q=1}^{m}
a_{qj}\varphi_q^k(Y(t))\right)X_j(t)\mathrm{d}t\\
&\quad+\sum\limits_{j=1}^{m}\int_0^{\theta_0\wedge\varepsilon}\exp(-rt)\beta(X(t),Y(t))\frac{\partial\varphi_j^k}{\partial y}(Y(t))X_j(t)\mathrm{d}B(t)\\
&\quad+\sum\limits_{j=1}^{m}\int_0^{\theta_0\wedge\varepsilon}\exp(-rt)\varphi_j^k(Y(t))\mathrm{d}M_j(t).
\end{aligned}
\end{equation}

By combining (\ref{e:DynProPri1}) and (\ref{e:DynIto}), it follows that
\[
\begin{aligned}
&\sum\limits_{j=1}^{m}\mathbb{E}\left[\frac{1}{\varepsilon}\left.\int_0^{\theta_0\wedge\varepsilon}\exp(-rt)\left(-f_j(Y(t))
+r\varphi_j^k(Y(t))\right.\right.\right.\\
&\quad-\frac{1}{2}\beta(X(t),Y(t))^2\frac{\partial^2\varphi_j^k}{\partial y^2}(Y(t))-\alpha(X(t),Y(t))\frac{\partial\varphi_j^k}{\partial y}(Y(t))\\
&\quad\left.\left.\left.-\sum\limits_{q=1}^{m}
a_{qj}\varphi_q^k(Y(t))\right)X_j(t)\mathrm{d}t
\right|X(0)=e_i,Y(0)=y_i\right]\geq-\frac{1}{k\varepsilon},
\end{aligned}
\]
i.e., as $X(t)=e_i$ for $t\in(0,\theta_0)$,
\[
\begin{aligned}
&\mathbb{E}\left[\frac{1}{\varepsilon}\left.\int_0^{\theta_0\wedge\varepsilon}\exp(-rt)\left(-f_i(Y(t))
+r\varphi_i(Y(t))-\mathcal{L}_i\varphi_i(Y(t))\right.\right.\right.\\
&\qquad\qquad\qquad\left.\left.\left.-\sum\limits_{q=1}^{m}a_{qi}\varphi_q^k(Y(t))\right)\mathrm{d}t
\right|X(0)=e_i,Y(0)=y_i\right]\geq-\frac{1}{k\varepsilon}.
\end{aligned}
\]
By sending $k\rightarrow+\infty$, it follows that
\[
\begin{aligned}
&\mathbb{E}\left[\frac{1}{\varepsilon}\left.\int_0^{\theta_0\wedge\varepsilon}\exp(-rt)\left(-f_i(Y(t))
+r\varphi_i(Y(t))-\mathcal{L}_i\varphi_i(Y(t))\right.\right.\right.\\
&\qquad\qquad\qquad\left.\left.\left.-\sum\limits_{q=1}^{m}a_{qi}V_q(Y(t))\right)\mathrm{d}t
\right|X(0)=e_i,Y(0)=y_i\right]\geq 0.
\end{aligned}
\]
Now taking limits in the above inequality as $\varepsilon\rightarrow 0^+$ and applying the mean value theorem, we obtain
\begin{equation}\label{e:SupeSolu1}
rV_i(y_i)-\mathcal{L}_i\varphi_i(y_i)-\sum\limits_{q=1}^{m} a_{qi}V_q(y_i)-f_i(y_i)\geq0.
\end{equation}

By the definition of $V_i$, we have
\begin{equation}\label{e:Supesolu2}
V_i(y_i)-g_i(y_i)\geq0.
\end{equation}

From (\ref{e:SupeSolu1}) and (\ref{e:Supesolu2}), we see that $V_i$ is a viscosity supersolution of the variational inequality $(\ref{e:VariIneq})$.

\noindent 3. \textit{The viscosity subsolution property.} Fix an $i\in \{1,2,\cdots,m\}$. Suppose that $\varphi_i\in C^2(\mathbb{R})$ and $V_i-\varphi_i$
has a maximum at some point $\bar{y}_i\in\mathbb{R}$ such that $V_i(\bar{y}_i)=\varphi_i(\bar{y}_i)$. We will prove that $V_i$ is a viscosity subsolution
of the variational inequality (\ref{e:VariIneq}) by contradiction.

Suppose that
\[
rV_i(\bar{y}_i)-\mathcal{L}_i\varphi_i(\bar{y}_i)-\sum\limits_{q=1}^{m} a_{qi}V_q(\bar{y}_i)-f_i(\bar{y}_i)>0,
\]
and
\[
V_i(\bar{y}_i)-g_i(\bar{y}_i)>0.
\]

Hence for any positive number $\varepsilon$ small enough and $k$ large enough, thanks to continuity of the functions $\mathcal{L}_i\varphi_i(\cdot)$,
$f_i(\cdot)$ and $V_j(\cdot)$, where $j=1,2,\cdots,m$, there is a positive number $\rho$ such that
\[
r\varphi_i(y)-\mathcal{L}_i\varphi_i(y)-\sum\limits_{q=1}^{m} a_{qi}\varphi_q^k(y)-f_i(y)>\varepsilon\;\mathrm{for}\;y\in(\bar{y}_i-\rho, \bar{y}_i+\rho),
\]
and
\[
\varphi_i(y)-g_i(y)>\varepsilon\;\mathrm{for}\;y\in(\bar{y}_i-\rho, \bar{y}_i+\rho).
\]

Consequently, for $X(0)=e_i$, $Y(0)=\bar{y}_i$, defining $\delta:=\inf\{t:t>0,X(t)\neq e_i,Y(t)\notin(\bar{y}_i-\rho,\bar{y}_i+\rho)\}$, we have
\begin{equation}\label{e:SubSolu1}
r\varphi_i(Y(t))-\mathcal{L}_i\varphi_i(Y(t))-\sum\limits_{q=1}^{m} a_{qi}\varphi_q^k(Y(t))-f_i(Y(t))>\varepsilon\;\mathrm{for}\;0\leq t\leq\delta,
\end{equation}
and
\begin{equation}\label{e:SubSolu2}
\varphi_i(Y(t))-g_i(Y(t))>\varepsilon\;\mathrm{for}\;0\leq t\leq\delta.
\end{equation}
.

For any stopping time $\tau\in\mathcal{T}$, by applying It\^{o}'s formula and noting (\ref{e:SubSolu1}) and (\ref{e:SubSolu2}), we obtain
\begin{equation}\label{e:ContIto}
\begin{aligned}
V_i(\bar{y}_i)&=\varphi_i(\bar{y}_i)\\
&=\mathbb{E}\left[\int_0^{\delta\wedge\tau}\exp(-rt)\left(
r\varphi_i(Y(t))-\mathcal{L}_i\varphi_i(Y(t))
-\sum\limits_{q=1}^{m} a_{qi}\varphi_q^k(Y(t))\right)\mathrm{d}t\right.\\
&\quad\left.\left.+\sum\limits_{j=1}^{m}\exp(-r(\delta\wedge\tau))\varphi_j^k(Y(\delta\wedge\tau))X_j(\delta\wedge\tau)
\right|X(0)=e_i,Y(0)=\bar{y}_i\right]\\
&\geq\mathbb{E}\left[\int_0^{\delta\wedge\tau}\exp(-rt)f(X(t),Y(t))\mathrm{d}t+\exp(-r\tau)g(X(\tau),Y(\tau))
\mathbbm{1}_{\{\delta>\tau\}}\right.\\
&\qquad\qquad+\left.\exp(-r\delta)V(X(\delta),Y(\delta))\mathbbm{1}_{\{\delta\leq\tau\}}|X(0)=e_i,Y(0)=\bar{y}_i\right]\\
&\quad+\varepsilon\mathbb{E}\left[\left.\int_0^{\delta\wedge\tau}\exp(-rt)\mathrm{d}t+\exp(-r\tau)\mathbbm{1}_{\{\delta>\tau\}}\right|X(0)=e_i,Y(0)=\bar{y}_i\right]\\
&\quad-\frac{1}{k}.
\end{aligned}
\end{equation}

We will show in Lemma \ref{l:Claim} ahead that there is a positive constant $C$ such that
\[
\mathbb{E}\left[\left.\int_0^{\delta\wedge\tau}\exp(-rt)\mathrm{d}t+\exp(-r\tau)\mathbbm{1}_{\{\delta>\tau\}}\right|X(0)=e_i,Y(0)=\bar{y}_i\right]\geq C.
\]

Sending $k\rightarrow\infty$ and then taking supremum over $\tau\in\mathcal{T}$ in (\ref{e:ContIto}), we find
\[
V_i(\bar{y}_i)\geq V_i(\bar{y}_i)+\varepsilon C,
\]
which is a contradiction. Consequently, $V_i$ is a viscosity subsolution of the variational inequality $(\ref{e:VariIneq})$.

\noindent 4. \textit{Uniqueness.} The proof of uniqueness is the same as that of \cite[pp.~98--99, Uniqueness property]{Pham2009}. Also refer to
\cite[pp.~31--32]{CIL92}.
\end{proof}

\begin{lem}\label{l:Claim}
Using the notations introduced in the proof of Theorem \ref{t:VisSol}, we have
\[
\mathbb{E}\left[\left.\int_0^{\delta\wedge\tau}\exp(-rt)\mathrm{d}t+\exp(-r\tau)\mathbbm{1}_{\{\delta>\tau\}}\right|X(0)=e_i,Y(0)=\bar{y}_i\right]\geq C
\]
for some positive constant $C$.
\end{lem}
\begin{proof}
1. Let $Y_i$ be the solution of the following equation,
\[
\mathrm{d}Y_i(t)=\alpha_i(Y_i(t))\mathrm{d}t+\beta_i(Y_i(t))\mathrm{d}B(t),\;\;Y_i(0)=\bar{y}_i.
\]

Define $T_{\bar{y}_i}:=\inf\{t: t>0, |Y_i(t)-\bar{y}_i|=\rho\}$. Since $\beta_i(\cdot)>0$, we have, by \cite[p.~152, Corollary 6.12.2]{Kle98},
\begin{equation}\label{e:Cross1}
\mathbb{P}(T_{\bar{y}_i}<+\infty|Y_i(0)=\bar{y}_i)>0.
\end{equation}

\noindent 2. Define $T:=\inf\{t:t>0, X(t)\neq e_i\}$. Thanks to the independence of $X$ and $B$, it follows that
\[
\begin{split}
&\qquad\mathbb{P}(\exists t\in(0, T),\; \mathrm{s.t.}\;|Y(t)-\bar{y}_i|=\rho|X(0)=e_i, Y(0)=\bar{y}_i)\\
&=\int_0^{+\infty}\mathbb{P}(\exists s\in(0, t),\; \mathrm{s.t.}\;|Y_i(s)-\bar{y}_i|=\rho|Y(0)=\bar{y}_i)(-a_{ii}\exp(a_{ii}t))\mathrm{d}t,
\end{split}
\]
where we have used the fact that $\mathbb{P}(T>t|X(0)=e_i)=\exp(a_{ii}t)$.

The above equality and (\ref{e:Cross1}) yield
\[
\mathbb{P}(\exists t\in(0, T),\; \mathrm{s.t.}\;|Y(t)-\bar{y}_i|=\rho|X(0)=e_i, Y(0)=\bar{y}_i)>0.
\]
Consequently, we obtain
\begin{equation}\label{e:GreatThan0}
\mathbb{E}\left[\left.|Y(\delta)-\bar{y}_i|^2\right|X(0)=e_i,Y(0)=\bar{y}_i\right]>0.
\end{equation}

\noindent 3. Set $B_i:=(\bar{y}_i-\rho,\bar{y}_i+\rho)$ and
\[
C_1:=\min\left\{\left(r+\frac{2}{\rho}\sup\limits_{y\in B_i}|\alpha(e_i,y)|
+\frac{1}{\rho^2}\sup\limits_{y\in B_i}\beta(e_i,y)^2\right)^{-1},1\right\}.
\]
Define a function $G:B_i\rightarrow \mathbb{R}$ by
\[
G(y):=C_1\left(1-\frac{|y-\bar{y}_i|^2}{\rho^2}\right).
\]

Then we have
\[
\begin{aligned}
&\quad\;\mathbb{E}\left[\left.\int_0^{\delta\wedge\tau}\exp(-rt)\mathrm{d}t+\exp(-r\tau)\mathbbm{1}_{\{\delta>\tau\}}\right|X(0)=e_i,Y(0)=\bar{y}_i\right]\\
&\geq\mathbb{E}\left[\left.\int_0^{\delta\wedge\tau}\exp(-rt)\left(rG(Y(t))-\mathcal{L}_iG(Y(t))\right)\mathrm{d}t\right.\right.\\
&\qquad\;+\exp(-r\delta\wedge\tau)G(Y(\delta\wedge\tau))\\
&\qquad\;\left.\left.-\exp(-r\delta)G(Y(\delta))\mathbbm{1}_{\{\delta\leq\tau\}}\right|X(0)=e_i,Y(0)=\bar{y}_i\right]\\
&=G(\bar{y}_i)-\mathbb{E}\left[\left.\exp(-r\delta)G(Y(\delta))\mathbbm{1}_{\{\delta\leq\tau\}}\right|X(0)=e_i,Y(0)=\bar{y}_i\right]\\
&\geq C_1\rho^{-2}\mathbb{E}\left[\left.|Y(\delta)-\bar{y}_i|^2\right|X(0)=e_i,Y(0)=\bar{y}_i\right]=:C>0,
\end{aligned}
\]
where we have used the facts $rG(y)-\mathcal{L}_iG(y)\leq1$ and $G(y)\leq1$ for the first inequality, It\^{o}'s formula for the first equality, and (\ref{e:GreatThan0}) for the last inequality. The proof is complete.
\end{proof}

We will prove that $V$ is uniquely determined by the system (\ref{e:VariIneq}). To do this, we introduce some notations.

Set $\mu_i:=\limsup\limits_{x\rightarrow\infty}\left|\frac{\alpha_i(x)}{x}\right|$ and
$\sigma_i:=\limsup\limits_{x\rightarrow\infty}\left|\frac{\beta_i(x)}{x}\right|$, $i=1,2,\cdots,m$. Then we define the functions
$w_i(\lambda):=r-a_{ii}-\mu_i\lambda-\frac{1}{2}\sigma_i^2\lambda(\lambda-1)$, $i=1,2,\cdots,m$, and a matrix $H(\lambda):=(h_{ij}(\lambda))_{m\times m}$, where
$h_{ii}:=w_i$ and $h_{ij}:=-a_{ij}$ for $i\neq j$.

We make the following hypothesis.
\begin{enumerate}
\makeatletter
\renewcommand{\labelenumi}{(H\theenumi)}\renewcommand{\p@enumi}{H}
\makeatother
\item\label{i:H1} There are positive numbers $b_i$'s such that $\sum\limits_{i=1}^mb_i h_{ij}(\overline{\lambda})\geq 0$ for some constant $\overline{\lambda}>1$, $j=1,2,\cdots,m$.
\end{enumerate}

\begin{thm}\label{t:Uniq1}
Assume that \eqref{i:H1} holds. Let $\{U_i\}_{i=1}^m$ be a family of functions defined on $\mathbb{R}$ with at most linear growth such
that for each fixed $i$, $U_i$ is a viscosity solution of
\begin{equation*}
\min\{rU_i-\mathcal{L}_iU_i-\sum\limits_{q=1}^{m}a_{qi}U_q-f_i, U_i-g_i\}=0\;\,\mathrm{on}\;\,\mathbb{R}.
\end{equation*}
Then $U_i=V_i$ for $i=1,2,\cdots,m$.
\end{thm}
\begin{proof}
1. Since $U_i$ and $V_i$, $i=1,2,\cdots,m$, grow at most linearly, we have
\[
\lim\limits_{x\rightarrow\infty}(U_i(x)-V_i(x)-b_i|x|^{\overline{\lambda}})=-\infty,\;i=1,2,\cdots,m.
\]

\noindent 2. Thanks to \eqref{i:H1}, there is a positive constant $K$ such that, for each $i=1,2,\cdots,m$,
\[
\min\{r\widetilde{\psi}_i-\mathcal{L}_i\widetilde{\psi}_i-\sum\limits_{q=1}^{m}a_{qi}\widetilde{\psi}_q, \widetilde{\psi}_i-g_i\}\geq0\;\;\mathrm{on}\;\,\mathbb{R}\setminus[-K,K],
\]
where $\widetilde{\psi}_i(x):=b_i|x|^{\overline{\lambda}}$.

Choose some functions $\overline{\psi}_i\in C^2(\mathbb{R})$ satisfying $\overline{\psi}_i|_{\mathbb{R}{\setminus[-K,K]}}=\widetilde{\psi}_i$, $i=1,2,\cdots,m$. Then it follows that
\[
\min\{r\psi_i-\mathcal{L}_i\psi_i-\sum\limits_{q=1}^{m}a_{qi}\psi_q, \psi_i-g_i\}\geq0\;\;\mathrm{on}\;\,\mathbb{R},
\]
where $\psi_i(x):=C+\overline{\psi}_i(x)$, $i=1,2,\cdots,m$, for some constant $C$ large enough.

\noindent 3. Define $V_i^\varepsilon:=V_i+\varepsilon\psi_i$ for $\varepsilon\in(0,1)$ and $i=1,2,\cdots,m$. Then $V_i^\varepsilon$ is a viscosity
supersolution of
\begin{equation*}
\min\{rV_i^\varepsilon-\mathcal{L}_iV_i^\varepsilon-\sum\limits_{q=1}^{m}a_{qi}V_q^\varepsilon-f_i, V_i^\varepsilon-g_i\}=0\;\,\mathrm{on}\;\,\mathbb{R}.
\end{equation*}

\noindent 4. We prove $U_i\leq V_i^\varepsilon$, $i=1,2,\cdots,m$, by contradiction. Otherwise, we have by Step 1
\begin{equation}\label{e:Maxi}
\begin{aligned}
M:&=\max\limits_{i\in\{1,\cdots,m\}}\sup\limits_{x\in\mathbb{R}}\{U_i(x)-V_i^\varepsilon(x)\}\\
&=\max\limits_{x\in\mathbb{R}}\{U_k(x)-V_k^\varepsilon(x)\}
=U_k(\hat{x})-V_k^\varepsilon(\hat{x})>0
\end{aligned}
\end{equation}
for some $k\in\{1,2,\cdots,m\}$ and $\hat{x}\in\mathbb{R}$.

For each positive integer $n\in\mathbb{N}$, we define a function $\Phi_n$ by
\[
\Phi_n(x,y):=U_k(x)-V_k^\varepsilon(y)-\frac{n}{2}(x-y)^2.
\]
In virtue of Step 1, we take $\mathcal{O}:=(-R,R)$ with $U_k(x)-V_k^\varepsilon(x)<0$ for all $|x|\geq R$. Then $\Phi_n$ has a maximum point $(x_n,y_n)$ on
$\overline{\mathcal{O}}$. By the similar argument to the step 1 of the proof of \cite[p.~77, Theorem 4.4.4]{Pham2009}, we have, up to a subsequence,
\begin{equation}\label{e:LimiZero}
\lim\limits_{n\rightarrow\infty}\Phi_n(x_n,y_n)=M,\;\lim\limits_{n\rightarrow\infty}n(x_n-y_n)^2=0\;\mathrm{and}\;
\lim\limits_{n\rightarrow\infty}x_n=\lim\limits_{n\rightarrow\infty}y_n=\hat{x}.
\end{equation}

Therefore, for $n$ large enough, $(x_n,y_n)$ is a local maximum point of $\Phi_n$ relative to $\mathcal{O}$. Consequently, according to \cite[p.~17, p.~19]{CIL92}, there are two constants $a$ and $b$ with $a\leq b$ such that
\[
(n(x_n-y_n),a)\in\overline{J}_\mathcal{O}^{2,+}U_k(x_n),\;(n(x_n-y_n),b)\in\overline{J}_\mathcal{O}^{2,-}V_k(x_n)
\]
and
\begin{equation}\label{E:IshiLemm}
a\beta_k(x_n)^2-b\beta_k(y_n)^2\leq 3n(\beta_k(x_n)-\beta_k(y_n))^2.
\end{equation}
Then, in light of the viscosity property of $U_k$ and $V_k^{\varepsilon}$, it follows that
\begin{equation}\label{e:SubVisUk}
\begin{aligned}
\min\{&rU_k(x_n)-\alpha_k(x_n)(n(x_n-y_n))-\frac{a}{2}\beta_k(x_n)^2\\
&-\sum\limits_{q=1}^{m}a_{qk}U_q(x_n)-f_k(x_n), U_k(x_n)-g_k(x_n)\}\leq0
\end{aligned}
\end{equation}
and
\begin{equation}\label{e:SupVisVk}
\begin{aligned}
\min\{&rV_k^\varepsilon(y_n)-\alpha_k(y_n)(n(x_n-y_n))-\frac{b}{2}\beta_k(y_n)^2\\
&-\sum\limits_{q=1}^{m}a_{qk}V_q^\varepsilon(y_n)-f_k(y_n),V^\varepsilon_k(y_n)-g_k(y_n)\}\geq0.
\end{aligned}
\end{equation}

\noindent \texttt{Case 1} $U_k(x_n)-g_k(x_n)>0$ for all $n$ large enough.

Subtracting (\ref{e:SupVisVk}) from (\ref{e:SubVisUk}) yileds
\[
\begin{split}
r(U_k(x_n)-V^\varepsilon_k(y_n))\leq &n(\alpha_k(x_n)-\alpha_k(y_n))(x_n-y_n)+\frac{a}{2}\beta_k(x_n)^2-\frac{b}{2}\beta_k(y_n)^2\\
&+\sum\limits_{q=1}^{m}a_{qk}(U_q(x_n)-V_q^\varepsilon(y_n))+f_k(x_n)-f_k(y_n)
\end{split}
\]
for all $n$ large enough.

Note (\ref{e:LimiZero}), (\ref{E:IshiLemm}) and the Lipschitz continuity of the functions $\alpha_k$, $\beta_k$ and $f_k$. By taking limits in the above
inequality as $n\rightarrow+\infty$, it follows that
\[
r(U_k(\hat{x})-V^\varepsilon_k(\hat{x}))\leq \sum\limits_{q=1}^{m}a_{qk}(U_q(\hat{x})-V_q^\varepsilon(\hat{x})).
\]

Therefore, in light of \eqref{e:Maxi} and $\sum\limits_{q=1}^{m}a_{qk}=0$, we have $rM\leq 0$, which is a contradiction.

\noindent \texttt{Case 2} $U_k(x_n)-g_k(x_n)\leq0$ frequently.

In this case, the following inequality
\[
U_k(x_n)-V_k^\varepsilon(y_n)\leq g_k(x_n)-g_k(y_n)
\]
holds frequently.

Then, using the continuity of the function $g_k$, we get $M\leq 0$, which is a contradiction.

In conclusion, we have proved $U_i\leq V_i^\varepsilon$, $i=1,2,\cdots,m$.

\noindent 5. Taking limits in $U_i\leq V_i^\varepsilon$, $i=1,2,\cdots,m$, as $\varepsilon\rightarrow 0^+$, we have $U_i\leq V_i$, $i=1,2,\cdots,m$.

\noindent 6. Similarly, $U_i\geq V_i$, $i=1,2,\cdots,m$. This and Step 5 imply $U_i=V_i$, $i=1,2,\cdots,m$.
\end{proof}

\begin{cor}\label{c:Uniq2}
Assume that
\begin{equation}\label{e:TechCond}
 r>\max\{\mu_1,\mu_2,\cdots,\mu_m\}.
\end{equation}
Let $U_i$, $i=1, 2,\cdots, m$ be a family of functions defined on $\mathbb{R}$ with at most linear growth such that for each $i$, $U_i$ is a viscosity
solution of
\begin{equation*}
\min\{rU_i-\mathcal{L}_iU_i-\sum\limits_{q=1}^{m}a_{qi}U_q-f_i, U_i-g_i\}=0\;\,\mathrm{on}\;\,\mathbb{R}.
\end{equation*}
Then $U_i=V_i$ for $i=1,2,\cdots,m$.
\end{cor}
\begin{proof}
Noting that $\sum\limits_{i=1}^m a_{ij}=0$, $j=1,2,\cdots,m$, we have $\sum\limits_{i=1}^m h_{ij}(\lambda)=r-\mu_j\lambda-\sigma_j^2\lambda(\lambda-1)$. This and \eqref{e:TechCond} yield \eqref{i:H1}.
\end{proof}

Let us introduce some notations as follows.
\[
\mathcal{C}_i:=\{y:V_i(y)>g_i(y)\},
\]
\[
\mathcal{S}_i:=\{y:V_i(y)=g_i(y)\},
\]
\[
\mathcal{D}_i:=\{y:rg_i(y)-\mathcal{L}_ig_i(y)-\sum\limits_{q=1}^ma_{qi}g_q(y)-f_i(y)\geq 0\},
\]
and
\[
\widehat{V}_i(y):=\mathbb{E}\left[\left.\int_{0}^{\infty}\exp(-rt)f(X(t),Y(t))\mathrm{d}t\right|X(0)=e_i,Y(0)=y\right],
\]
for each $i=1,2,\cdots,m$.

We consider the regularity of $V$ in the following theorem.
\begin{thm}\label{t:Regu}
The function $V_i$ is $C^2$ continuous on $\mathcal{C}_i$ and $C^1$ continuous on
$\partial \mathcal{C}_i$, $i=1,2,\cdots,m$.
\end{thm}
\begin{proof}
It follows from $\beta(\cdot,\cdot)>0$ that $\mathcal{L}_i$'s are locally elliptic. Then, thanks to Theorem \ref{t:VisSol}, the proof is completed in the same way as that of \cite[p.~100, Lemma 5.2.2, Proposition 5.2.1]{Pham2009}.
\end{proof}

\begin{thm}
The stopping time $\tau^*:=\inf\{t>0:(X(t),Y(t))\in\bigcup\limits_{i=1}^m\{e_i\}\times\mathcal{S}_i\}$ is an optimal stopping time of the problem \eqref{e:genopt}.
\end{thm}
\begin{proof}
We refer to \cite[pp.~101--102]{Pham2009} for the proof.
\end{proof}

Theorem \ref{t:OptiStra} states some properties about $\mathcal{S}_i$'s.
\begin{thm}\label{t:OptiStra}
Assume that \eqref{i:H1} holds. Then the following conclusions are true.

\noindent$\mathrm{(1)}$ If $\bigcup\limits_{i=1}^m\mathcal{S}_i=\emptyset$, then $\widehat{V}_i\geq g_i$ for each $i=1,2,\cdots,m$.

\noindent$\mathrm{(2)}$ If $\widehat{V}_i\geq g_i$ for each $i=1,2,\cdots,m$, then $V_i=\widehat{V}_i$.

\noindent$\mathrm{(3)}$ If $g_i$ is $C^2$ continuous for some $i\in\{1,2,\cdots,m\}$, then $\mathcal{S}_i$ is included in $\mathcal{D}_i$.
\end{thm}
\begin{proof}
By virtue of Theorem \ref{t:Uniq1}, (1) and (2) are proved in a similar way to that of \cite[p.~102, Lemma 5.2.3]{Pham2009}; for the proof of (3), we apply Theorem \ref{t:Uniq1} and refer to \cite[p.~102, Lemma 5.2.4]{Pham2009}.
\end{proof}

Now we study the forms of $\mathcal{S}_i$'s. To do this, we make the following assumption.
\begin{enumerate}
\makeatletter
\renewcommand{\labelenumi}{(H\theenumi)}\renewcommand{\p@enumi}{H}
\makeatother
\setcounter{enumi}{+1}
\item\label{i:H2} The process $Y$ takes values in $(0,+\infty)$ and $|\alpha_i(y)|+|\beta_i(y)|\leq C|y|$ for some constant $C$, $i=1,2,\cdots,m$.
\end{enumerate}

We will use the fact that the matrix $rI_{m\times m}-(a_{ij})_{m\times m}$ is invertible. This follows from that $rI_{m\times m}-(a_{ij})_{m\times m}$ is
a strictly diagonally dominant matrix, since $r>0$ and $\sum\limits_{i=1}^ma_{ij}=0$ for all $j=1,2,\cdots,m$. Let $(b_{ij})_{m\times m}$ denote the
inverse of $rI_{m\times m}-(a_{ij})_{m\times m}$.

Since $V_i$'s are Lipschitz continuous, it is reasonable to set $V_i(0):=V_i(0^+)$.
\begin{lem}\label{l:ViscZero}
Assume that \eqref{i:H2} holds. Then
\[
V_i(0)=\max\{\sum\limits_{q=1}^mb_{qi}f_q(0),g_i(0)\},\quad i=1,2,\cdots,m.
\]
\end{lem}
\begin{proof}
By taking limits as $y\rightarrow 0^+$, Theorem \ref{t:VisSol} implies that $V_i(0)$ is a solution of
\[
\min\{rV_i(0)-\sum\limits_{q=1}^{m}a_{qi}V_q(0)-f_i(0), V_i(0)-g_i(0)\}=0.
\]
Consequently, $V_i(0)=\max\{\sum\limits_{q=1}^mb_{qi}f_q(0),g_i(0)\}.$ This completes the proof.
\end{proof}

\begin{thm}\label{t:OptiStraForm}
Assume that \eqref{i:H1} and \eqref{i:H2} hold, $\mathcal{S}_i$'s are nonempty connected sets, and $g_i$'s are $C^2$ continuous. Then the following are true.\\
$\mathrm{(1)}$ If $\mathcal{D}_i=[a_i,+\infty)$ for some positive constant $a_i$, $i=1,2,\cdots,m$, then $\mathcal{S}_i=[\inf\mathcal{S}_i,+\infty)$.\\
$\mathrm{(2)}$ If $g_i(0)>\sum\limits_{i=1}^mb_{qi}f_q(0)$ and $\mathcal{D}_i=(0,b_i]$ for some positive constant $b_i$, $i=1,2,\cdots,m$, then
$\mathcal{S}_i=(0,\sup\mathcal{S}_i]$.
\end{thm}

\begin{proof}
Set $a_i^*:=\inf\mathcal{S}_i$ in Case (1) and $b_i^*:=\sup\mathcal{S}_i$ in Case (2). Then we define $A_i:=[a_i^*,+\infty)$ in Case (1) and $A_i:=(0,b_i^*]$ in Case (2). We will prove $\mathcal{S}_i=A_i$.

Note the following facts,
\begin{enumerate}[(i)]
\item $g_i$'s grow at most linearly;
\item $g_i(a_i^*)=V_i(a_i^*)$ or $g_i(b_i^*)=V_i(b_i^*)$ by the continuity of $g_i$ and $V_i$;
\item in light of Lemma \ref{l:ViscZero}, $g_i(0)=V_i(0)$ if $g_i(0)>\sum\limits_{i=1}^mb_{qi}f_q(0)$ and $\mathcal{D}_i=(0,b_i]$.
\end{enumerate}
Moreover, thanks to $A_i\subset\mathcal{D}_i$ by Theorem \ref{t:OptiStra}, $g_i$ solves the inequality
\[
rg_i-\mathcal{L}_ig_i-\sum\limits_{q=1}^{m}a_{qi}g_q-f_i\geq0\;\,\text{on $A_i$.}
\]

Thus, by Theorem \ref{t:Uniq1}, it follows that $g_i=V_i$ on $A_i$. The proof is complete.
\end{proof}

\section{An Application}\label{model}
A firm is extracting a kind of natural resource (oil, gas, etc.). We will determine what time is optimal to stop the extraction.

Let $(\Omega, \mathscr{F}, \{\mathscr{F}_t\}_{t\geq
0}, \mathbb{P})$ be a filtered probability space. We assume that $\{\mathscr{F}_t\}_{t\geq 0}$ satisfies the usual conditions and $\mathscr{F}_0$ is the
completion of $\{\emptyset, \Omega\}$.

Let $X:=(X(t), t\geq 0)$ be a time homogeneous Markov chain defined on $(\Omega, \mathscr{F}, \{\mathscr{F}_t\}_{t\geq
0}, \mathbb{P})$ taking values in the standard orthogonal basis of $\mathbb{R}^2$,  $I:=\{e_1, e_2\}$, with a rate matrix $A:=\left[
\begin{array}{cc}
    -\lambda_1 & \lambda_2  \\
    \lambda_1 & -\lambda_2 \\
    \end{array}
\right] $ for some positive constants $\lambda_1, \lambda_2$.

Let $B(t)$ be a one dimensional standard Brownian motion, which is independent of $X$, defined on $(\Omega, \mathscr{F}, \{\mathscr{F}_t\}_{t\geq 0},
\mathbb{P})$.

We assume that the price process $P$ of the resource satisfies
\begin{equation}\label{e:price1}
\text{$\mathrm{d}P(t)=\mu(X(t)) P(t)\mathrm{d}t+\sigma(X(t)) P(t)\mathrm{d}B(t)$ and $P(0)>0$,}
\end{equation}
where $\mu(e_1):=\mu_1$, $\mu(e_2):=\mu_2$, $\sigma(e_1):=\sigma_1>0$ and $\sigma(e_2):=\sigma_2>0$. We will assume that $\mu_1\leq \mu_2$.

Applying It\^{o}'s formula, we deduce that the solution of Equation (\ref{e:price1}) is
\begin{equation}\label{e:price2}
P(t)=P(0)\exp\left[\int_0^t\left(\mu(X(s))-\frac{1}{2}\sigma(X(s))^2\right)\mathrm{d}s+\int_0^t\sigma(X(s))\mathrm{d}B(s)\right].
\end{equation}

To answer the question--- what time is optimal to stop the extraction, we will solve the following optimal
problem,
\begin{equation}\label{e:optimal1}
\begin{aligned}
J_i(p):=\sup\limits_{\tau\in\mathcal{T}}\mathbb{E}&\left[\left.\int_0^{\tau}\exp(-rt)(P(t)-C)\mathrm{d}t\right.\right.\\
&\left.\left.-\exp(-r\tau)K\right|X(0)=e_i,P(0)=p\right],
\end{aligned}
\end{equation}
where $\mathcal{T}$ is the collection of all stopping times, $r$ is the discount rate, $C$ is the running cost rate with $C>0$, and $K$ is the cost at which the firm stops the extraction.

\begin{lem}[{\cite[Lemma 1]{GuoZhang2005}}]\label{l:Expec}
we have
\[
\mathbb{E}\left[\left.\exp\left(\int_0^t\mu(X(s))\mathrm{d}s\right)\right|X(0)=e_i\right]=\frac{x_2-\mu_i}{x_2-x_1}\exp(x_1t)+\frac{\mu_i-x_1}{x_2-x_1}\exp(x_2t),
\]
where $x_1$ and $x_2$ are the solutions of
\begin{equation*}
x^2+(\lambda_1-\mu_1+\lambda_2-\mu_2)x+(\lambda_1-\mu_1)(\lambda_2-\mu_2)-\lambda_1\lambda_2=0
\end{equation*}
with $x_1<x_2$.
\end{lem}
\begin{rem}
If $\mu_1<\mu_2$, then $x_1<\mu_1<x_2<\mu_2$; if $\mu_1=\mu_2$, then $x_1=\mu_1-\lambda_1-\lambda_2$ and $x_2=\mu_1$.
\end{rem}

\begin{thm}\label{t:optimal1}
Assume $r\leq x_2$. Then the optimal stopping time $\tau^*$ is given by $\tau^*=+\infty$ a.s., and the function $J_i$ in \eqref{e:optimal1} is given by
$J_i(p)=+\infty$, $i=1,2$.
\end{thm}
\begin{proof}
By \eqref{e:price2} and (\ref{e:optimal1}), we have, for fixed $T>0$,
\[
\begin{aligned}
J_i(p)
\geq &p\int_0^T\mathbb{E}\left[\exp(-rt)\left.\exp\left(\int_0^t\mu(X(s))-\frac{1}{2}\sigma(X(s))^2\mathrm{d}s\right.\right.\right.\\
&\qquad\qquad\left.\left.\left.+\int_0^t\sigma(X(s))\mathrm{d}B(s)\right)\right|X(0)=e_i\right]\mathrm{d}t\\
&-r^{-1}(1-\exp(-rT))C-\exp(-rT)K.
\end{aligned}
\]

Note that
\[
\left(\exp\left(\int_0^t\sigma(X(s))\mathrm{d}B(s)-\frac{1}{2}\int_0^t\sigma(X(s))^2\mathrm{d}s)\right),t\geq 0\right)
\]
is a martingale. Then in light of the independence of $X$ and $B$, we get
\[
\begin{aligned}
J_i(p)&\geq p\int_0^T\exp(-rt)\mathbb{E}\left[\left.\exp\left(\int_0^t\mu(X(s))\right)\mathrm{d}s\right|X(0)=e_i\right]\mathrm{d}t\\
&\qquad-r^{-1}(1-\exp(-rT))C-\exp(-rT)K\\
&=p\left(\frac{x_2-\mu_i}{x_2-x_1}\int_0^T\exp((x_1-r)t)\mathrm{d}t+\frac{\mu_i-x_1}{x_2-x_1}\int_0^T\exp((x_2-r)t)\mathrm{d}t\right)\\
&\qquad-r^{-1}(1-\exp(-rT))C-\exp(-rT)K,
\end{aligned}
\]
where we have used Lemma \ref{l:Expec}.

Sending $T\rightarrow+\infty$ in the above inequality, we obtain $J_i(p)=+\infty$ since $r\leq x_2$.
\end{proof}

\begin{cor}
If $r\leq \mu_1$, then the conclusions in Theorem \ref{t:optimal1} hold.
\end{cor}
\begin{proof}
The proof can be completed by the fact that $r\leq \mu_1$ implies $r\leq x_2$.
\end{proof}

In the following discussion, we consider the case $r>x_2$.
\begin{lem}\label{l:Soluz}
$\mathrm{(1)}$ {\rm\cite[Remark 2.1]{Guo2001}} The following equation
\begin{equation*}
w_1(z)w_2(z)-\lambda_1\lambda_2=0,
\end{equation*}
where $w_i(z):=r+\lambda_i-\mu_iz-\frac{\sigma_i^2}{2}z(z-1)$, has four distinct solutions $z_i$'s satisfying $z_1<z_2<0<z_3<z_4$ with $w_1(z_3)>0$.

\noindent $\mathrm{(2)}$ Assume that $r>x_2$. Then $z_3>1$ and there are positive constants $b_i$'s fulfilling
$b_1w_1(z_3)-\lambda_1b_2=0$ and $b_2w_2(z_3)-\lambda_2b_1=0$. See Figure \ref{f:Soluz}.
\end{lem}
\begin{figure}[!h]
\centering
\includegraphics[height=6cm]{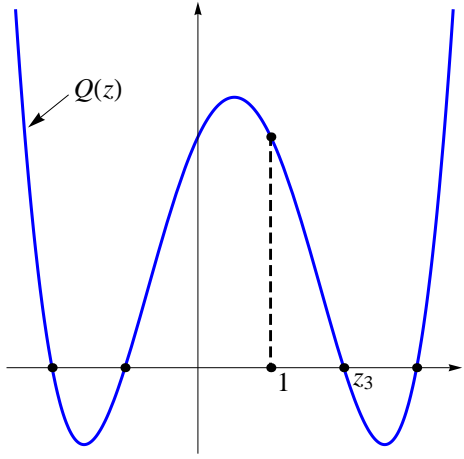}
\caption{Solutions of $w_1(z)w_2(z)-\lambda_1\lambda_2=0$.}
\label{f:Soluz}
\end{figure}
\begin{proof}
1. Set $Q(z):=w_1(z)w_2(z)-\lambda_1\lambda_2$. Note that the equation $w_1(z)=0$ has two solutions, say, $z^1$ and $z^2$ with $z^1<z^2$. Then $Q(z^i)<0$,
where $i=1,2$. In addition, we have $Q(0)>0$ and $\lim\limits_{z\rightarrow\infty}Q(z)=+\infty$. Therefore, by the intermediate value theorem, it follows
that the equation $w_1(z)w_2(z)-\lambda_1\lambda_2=0$ has four distinct solutions $z_i$'s satisfying $z_1<z_2<0<z_3<z_4$. Furthermore, $w_1(z_3)>0$, since $z^1<z_3<z^2$.

\noindent 2. Since $r>x_2$, we have
\begin{equation}\label{e:MarkEqua1}
r^2+(\lambda_1-\mu_1+\lambda_2-\mu_2)r+(\lambda_1-\mu_1)(\lambda_2-\mu_2)-\lambda_1\lambda_2>0
\end{equation}
and
\begin{equation}\label{e:MarkEqua2}
2r+\lambda_1-\mu_1+\lambda_2-\mu_2>0.
\end{equation}

Then it follows from (\ref{e:MarkEqua1}) that $Q(1)>0$ and $(r+\lambda_1-\mu_1)(r+\lambda_2-\mu_2)>0$. Combining (\ref{e:MarkEqua2}) and
$(r+\lambda_1-\mu_1)(r+\lambda_2-\mu_2)>0$, we have $r+\lambda_1-\mu_1>0$ and $r+\lambda_2-\mu_2>0$.

\noindent 3. In this step, we prove $z_3>1$.

\texttt{Case 1} $\mu_i+\sigma_i^2/2>0$, where $i=1,2$.

In this case, we have
\[
Q'(1)=-(r+\lambda_1-\mu_1)(\mu_2+\sigma_2^2/2)-(r+\lambda_2-\mu_2)(\mu_1+\sigma_1^2/2)<0.
\]
This and $Q(1)>0$ imply $z_3>1$.

\texttt{Case 2} $\mu_i+\sigma_i^2/2<0$, where $i=1,2$.

In this case, we have
\[
Q'''(1)=3\sigma_1^2\sigma_2^2+3(\sigma_1^2\mu_2+\sigma_2^2\mu_1)<0.
\]
This and $Q(1)>0$ imply $z_3>1$.

\texttt{Case 3} $(\mu_1+\sigma_1^2/2)(\mu_2+\sigma_2^2/2)\leq0$.

In this case, we have
\[
Q''(1)=-\sigma_1^2(r+\lambda_2-\mu_2)+2(\mu_1+\frac{\sigma_1^2}{2})(\mu_2+\frac{\sigma_2^2}{2})-\sigma_1^2(r+\lambda_2-\mu_2)<0.
\]
This and $Q(1)>0$ imply $z_3>1$.

\noindent 4. Take $b_1=1$ and $b_2=w_1(z_3)/\lambda_1$. Then we have $b_1w_1(z_3)-\lambda_1b_2=0$ and $b_2w_2(z_3)-\lambda_2b_1=0$.
\end{proof}

\begin{lem}\label{l:Conn}
Assume that $r>x_2$. Then $J_i$ is a Lipschitz continuous function, $i=1,2$. In addition, $\mathcal{S}_i:=\{p:p\in(0,+\infty)\;\mathrm{and}\;J_i(p)=-K\}$ is
connected, $i=1,2$.
\end{lem}
\begin{proof}
1. By \eqref{e:price2} and (\ref{e:optimal1}), we have
\[
\begin{aligned}
&\quad\;|J_i(p_1)-J_i(p_2)|\\
&\leq|p_1-p_2|\int_0^{+\infty}\mathbb{E}\left[\exp(-rt)\left.\exp\left(\int_0^t\mu(X(s))-\frac{1}{2}\sigma(X(s))^2\mathrm{d}s\right.\right.\right.\\
&\qquad\qquad\qquad\qquad\qquad\qquad\qquad\qquad\left.\left.\left.+\int_0^t\sigma(X(s))\mathrm{d}B(s)\right)\right|X(0)=e_i\right]\mathrm{d}t.
\end{aligned}
\]

Note that
\[
\left(\exp\left(\int_0^t\sigma(X(s))\mathrm{d}B(s)-\frac{1}{2}\int_0^t\sigma(X(s))^2\mathrm{d}s)\right),t\geq 0\right)
\]
is a martingale. Then in light of the independence of $X$ and $B$, we get
\[
\begin{aligned}
&\quad\;|J_i(p_1)-J_i(p_2)|\\
&\leq|p_1-p_2|\int_0^{+\infty}\exp(-rt)\mathbb{E}\left[\left.\exp\left(\int_0^t\mu(X(s))\right)\mathrm{d}s\right|X(0)=e_i\right]\mathrm{d}t\\
&=\frac{1}{x_2-x_1}\cdot\left(\frac{x_2-\mu_i}{r-x_1}+\frac{\mu_i-x_1}{r-x_2}\right)|p_1-p_2|,
\end{aligned}
\]
where we have used Lemma \ref{l:Expec} and $r>x_2$ for the equality.

\noindent 2. We prove that $\mathcal{S}_i:=\{p:p\in(0,+\infty)\;\mathrm{and}\;J_i(p)=-K\}$ is connected by contradiction. Suppose that $a, b\in\mathcal{S}_i$ but there is a point $c\in (a, b)$ with $c\notin \mathcal{S}_i$ (i.e., $J_i(c)>-K$).

Noting that $J_i$ is convex, we have
\[
J_i(c)\leq \frac{b-c}{b-a}J_i(a)+\frac{c-a}{b-a}J_i(b)=-K,
\]
which contradicts $J_i(c)>-K$.
\end{proof}

\begin{thm}\label{t:ViscSpec}
Assume $r>x_2$. Then $\{J_1,J_2\}$ is the unique solution with at most linear growth of the following variational inequalities
\[
\left\{
        \begin{array}{ll}
\min\{(r+\lambda_1)J_1(p)-\mu_1pJ_1'(p)-\frac{1}{2}\sigma_1^2p^2J_1''(p)\\
\qquad\qquad\qquad-\lambda_1J_2(p)-p+C,J_1(p)+K\}=0\\
\min\{(r+\lambda_2)J_2(p)-\mu_2pJ_2'(p)-\frac{1}{2}\sigma_2^2p^2J_2''(p)\\
\qquad\qquad\qquad-\lambda_2J_1(p)-p+C,J_2(p)+K\}=0
\end{array}
\right.
\]
with boundary condition $J_i(0)=-\min\{C/r,K\}$, $i=1,2$, on $(0,+\infty)$ in viscosity sense .
\end{thm}
\begin{proof}
It follows from Lemma \ref{l:ViscZero} that $J_i(0)=-\min\{C/r,K\}$, $i=1,2$. Then, by Lemma \ref{l:Soluz} and Lemma \ref{l:Conn}, the result is a straight corollary of Theorem \ref{t:Uniq1}.
\end{proof}

\begin{thm}\label{t:optimal2}
Assume that $r>x_2$ and $C\leq rK$. Then the optimal stopping time $\tau^*$ is given by $\tau^*=+\infty$ a.s.; furthermore, the functions $J_i$'s are
given by
\[
J_1(p)=k_2p-r^{-1}C,\;J_2(p)=k_1p-r^{-1}C,
\]
where
\[
k_i=\frac{r+\lambda_1+\lambda_2-\mu_i}{(r+\lambda_1-\mu_1)(r+\lambda_2-\mu_2)-\lambda_1\lambda_2},\;i=1,2.
\]
\end{thm}
\begin{proof}
Note that, for $i=1,2$,
\begin{equation*}
\mathcal{S}_i\subset\{p:0<p\leq C-rK\}
\end{equation*}
by Theorem \ref{t:OptiStra}. Thus we have $\mathcal{S}_i=\emptyset$.

Therefore, it follows from Theorem \ref{t:ViscSpec} that
\[
J_1(p)=\frac{(r+\lambda_1+\lambda_2-\mu_2)p}
{(r+\lambda_1-\mu_1)(r+\lambda_2-\mu_2)-\lambda_1\lambda_2}-\frac{C}{r}
\]
and
\[
J_2(p)=\frac{(r+\lambda_1+\lambda_2-\mu_1)p}
{(r+\lambda_1-\mu_1)(r+\lambda_2-\mu_2)-\lambda_1\lambda_2}-\frac{C}{r}.
\]
The proof is complete.
\end{proof}

\begin{cor}
If $r>\mu_2$ and $C\leq rK$, then the conclusions in Theorem \ref{t:optimal2} hold.
\end{cor}
\begin{proof}
The proof can be completed by the fact that $r>\mu_2$ implies $r>x_2$.
\end{proof}

Recall that $z_1$ and $z_2$ are the solutions of $w_1(z)w_2(z)-\lambda_1\lambda_2=0$ with $z_1<z_2<0$ (see Lemma \ref{l:Soluz}).
\begin{thm}\label{t:optimal3}
Assume that $r>x_2$ and $C>rK$. Then one and only one of the following holds.\\
$\mathrm{(1)}$ The equation
\[
\left[
\begin{array}{cc}
{p_1}^{-y_1}&0\\
0&{p_1}^{-y_2}
\end{array}
\right]
\left[
\begin{array}{c}
a_{11}p_1+b_{11}\\
a_{21}p_1+b_{21}
\end{array}
\right]=
\left[
\begin{array}{cc}
{p_2}^{-y_1}&0\\
0&{p_2}^{-y_2}
\end{array}
\right]
\left[
\begin{array}{c}
a_{12}p_2+b_{12}\\
a_{22}p_2+b_{22}
\end{array}
\right]
\]
has a solution $\{p_1^*,p_2^*\}$ with $p_1^*<p_2^*$,
where $y_1$ and $y_2$ are the two solutions of the following quadratic equation
\[
r+\lambda_1-\mu_1y-\frac{1}{2}\sigma_1^2y(y-1)=0
\]
with $y_1<y_2$,
\[
\left[
\begin{array}{c}
a_{11}\\
a_{21}
\end{array}
\right]=
\left[
\begin{array}{cc}
1&1\\
y_1&y_2
\end{array}
\right]^{-1}
\left[
\begin{array}{c}
\frac{-1}{r+\lambda_1-\mu_1}\\
\frac{-1}{r+\lambda_1-\mu_1}
\end{array}
\right],\;
\left[
\begin{array}{c}
b_{11}\\
b_{12}
\end{array}
\right]=
\left[
\begin{array}{cc}
1&1\\
y_1&y_2
\end{array}
\right]^{-1}
\left[
\begin{array}{c}
\frac{C-rK}{r+\lambda_1}\\
0
\end{array}
\right],
\]
\[
\begin{split}
\left[
\begin{array}{c}
a_{12}\\
a_{22}
\end{array}
\right]=&
\left[
\begin{array}{cc}
1&1\\
y_1&y_2
\end{array}
\right]^{-1}
\left(
\left[
\begin{array}{c}
k_2-\frac{1}{r+\lambda_1-\mu_1}\\
k_2-\frac{1}{r+\lambda_1-\mu_1}
\end{array}
\right]\right.\\
&-\lambda_1\left.
\left[
\begin{array}{cc}
1&1\\
z_1&z_2
\end{array}
\right]
\left[
\begin{array}{cc}
w_1(z_1)&w_1(z_2)\\
w_1(z_1)z_1&w_1(z_2)z_2
\end{array}
\right]^{-1}
\left[
\begin{array}{c}
k_1\\
k_1
\end{array}
\right]
\right),
\end{split}
\]
and
\[
\begin{split}
\left[
\begin{array}{c}
b_{12}\\
b_{22}
\end{array}
\right]=&
\left[
\begin{array}{cc}
1&1\\
y_1&y_2
\end{array}
\right]^{-1}
\left(
\left[
\begin{array}{c}
\frac{\lambda_1K+C}{r+\lambda_1}-\frac{C}{r}\\
0
\end{array}
\right]\right.\\
&+\lambda_1\left.
\left[
\begin{array}{cc}
1&1\\
z_1&z_2
\end{array}
\right]
\left[
\begin{array}{cc}
w_1(z_1)&w_1(z_2)\\
w_1(z_1)z_1&w_1(z_2)z_2
\end{array}
\right]^{-1}
\left[
\begin{array}{c}
\frac{C-rK}{r}\\
0
\end{array}
\right]
\right).
\end{split}
\]
In addition, the optimal stopping time $\tau^*$ is given by
\[
\tau^*=\inf\{t:t>0, (X(t),P(t))\in\{e_1\}\times(0,p_1^*]\cup\{e_2\}\times(0,p_2^*]\},
\]
and the $J_i$'s are given by
\[
J_1(p)=\left\{
\begin{array}{ll}
-K &\mathrm{on}\;(0,p_1^*]\\
\dfrac{p}{r+\lambda_1-\mu_1}-\dfrac{\lambda_1K+C}{r+\lambda_1}+A_1p^{y_1}+A_2p^{y_2} &\mathrm{on}\;(p_1^*,p_2^*]\\
k_2p-r^{-1}C+B_1p^{z_1}+B_2p^{z_2} &\mathrm{on}\;(p_2^*,+\infty)
\end{array}
\right.
\]
and
\[
J_2(p)=\left\{
\begin{array}{ll}
-K &\mathrm{on}\;(0,p_2^*]\\
k_1p-r^{-1}C+\lambda_1^{-1}w_1(z_1)B_1p^{z_1}+\lambda_1^{-1}w_1(z_2)B_2p^{z_2} &\mathrm{on}\;(p_2^*,+\infty),
\end{array}
\right.
\]
respectively. Here
\[
\left[
\begin{array}{c}
A_1\\
A_2
\end{array}
\right]=
\left[
\begin{array}{cc}
{p_1^*}^{-y_1}&0\\
0&{p_1^*}^{-y_2}
\end{array}
\right]
\left[
\begin{array}{c}
a_{11}p_1^*+b_{11}\\
a_{21}p_1^*+b_{21}
\end{array}
\right],
\]
\[
\left[
\begin{array}{c}
B_1\\
B_2
\end{array}
\right]=
\lambda_1\left[
\begin{array}{cc}
{p_2^*}^{-z_1}&0\\
0&{p_2^*}^{-z_2}
\end{array}
\right]
\left[
\begin{array}{cc}
w_1(z_1)&w_1(z_2)\\
w_1(z_1)z_1&w_1(z_2)z_2
\end{array}
\right]^{-1}
\left[
\begin{array}{c}
-k_1p_2^*+\frac{C-rK}{r}\\
-k_1p_2^*
\end{array}
\right].
\]
$\mathrm{(2)}$ The equation
\[
\left[
\begin{array}{cc}
{p_2}^{-\overline{y}_1}&0\\
0&{p_2}^{-\overline{y}_2}
\end{array}
\right]
\left[
\begin{array}{c}
\overline{a}_{11}p_2+\overline{b}_{11}\\
\overline{a}_{21}p_2+\overline{b}_{21}
\end{array}
\right]=
\left[
\begin{array}{cc}
{p_1}^{-\overline{y}_1}&0\\
0&{p_1}^{-\overline{y}_2}
\end{array}
\right]
\left[
\begin{array}{c}
\overline{a}_{12}p_1+\overline{b}_{12}\\
\overline{a}_{22}p_1+\overline{b}_{22}
\end{array}
\right]
\]
has a solution $\{\overline{p}_1^*,\overline{p}_2^*\}$ with $\overline{p}_1^*>\overline{p}_2^*$,
where $\overline{y}_1$ and $\overline{y}_2$ are the two solutions of the following quadratic equation
\[
r+\lambda_2-\mu_2y-\frac{1}{2}\sigma_2^2y(y-1)=0
\]
with $\overline{y}_1<\overline{y}_2$,
\[
\left[
\begin{array}{c}
\overline{a}_{11}\\
\overline{a}_{21}
\end{array}
\right]=
\left[
\begin{array}{cc}
1&1\\
\overline{y}_1&\overline{y}_2
\end{array}
\right]^{-1}
\left[
\begin{array}{c}
\frac{-1}{r+\lambda_2-\mu_2}\\
\frac{-1}{r+\lambda_2-\mu_2}
\end{array}
\right],\;
\left[
\begin{array}{c}
\overline{b}_{11}\\
\overline{b}_{12}
\end{array}
\right]=
\left[
\begin{array}{cc}
1&1\\
\overline{y}_1&\overline{y}_2
\end{array}
\right]^{-1}
\left[
\begin{array}{c}
\frac{C-rK}{r+\lambda_2}\\
0
\end{array}
\right],
\]
\[
\begin{split}
\left[
\begin{array}{c}
\overline{a}_{12}\\
\overline{a}_{22}
\end{array}
\right]=&
\left[
\begin{array}{cc}
1&1\\
\overline{y}_1&\overline{y}_2
\end{array}
\right]^{-1}
\left(
\left[
\begin{array}{c}
k_1-\frac{1}{r+\lambda_2-\mu_2}\\
k_1-\frac{1}{r+\lambda_2-\mu_2}
\end{array}
\right]\right.\\
&-\lambda_2\left.
\left[
\begin{array}{cc}
1&1\\
z_1&z_2
\end{array}
\right]
\left[
\begin{array}{cc}
w_2(z_1)&w_2(z_2)\\
w_2(z_1)z_1&w_2(z_2)z_2
\end{array}
\right]^{-1}
\left[
\begin{array}{c}
k_2\\
k_2
\end{array}
\right]
\right),
\end{split}
\]
and
\[
\begin{split}
\left[
\begin{array}{c}
\overline{b}_{12}\\
\overline{b}_{22}
\end{array}
\right]=&
\left[
\begin{array}{cc}
1&1\\
\overline{y}_1&\overline{y}_2
\end{array}
\right]^{-1}
\left(
\left[
\begin{array}{c}
\frac{\lambda_2K+C}{r+\lambda_2}-\frac{C}{r}\\
0
\end{array}
\right]\right.\\
&+\lambda_2\left.
\left[
\begin{array}{cc}
1&1\\
z_1&z_2
\end{array}
\right]
\left[
\begin{array}{cc}
w_2(z_1)&w_2(z_2)\\
w_2(z_1)z_1&w_2(z_2)z_2
\end{array}
\right]^{-1}
\left[
\begin{array}{c}
\frac{C-rK}{r}\\
0
\end{array}
\right]
\right).
\end{split}
\]
In addition, the optimal stopping time $\tau^*$ is given by
\[
\tau^*=\inf\{t:t>0, (X(t),P(t))\in\{e_1\}\times(0,\overline{p}_1^*]\cup\{e_2\}\times(0,\overline{p}_2^*]\},
\]
and the $J_i$'s are given by
\[
J_1(p)=\left\{
\begin{array}{ll}
-K &\mathrm{on}\;(0,\overline{p}_1^*]\\
k_2p-r^{-1}C+\lambda_2^{-1}w_2(z_1)\overline{B}_1p^{z_1}+\lambda_2^{-1}w_2(z_2)\overline{B}_2p^{z_2} &\mathrm{on}\;(\overline{p}_1^*,+\infty)
\end{array}
\right.
\]
and
\[
J_2(p)=\left\{
\begin{array}{ll}
-K &\mathrm{on}\;(0,\overline{p}_2^*]\\
\dfrac{p}{r+\lambda_2-\mu_2}-\dfrac{\lambda_2K+C}{r+\lambda_2}+\overline{A}_1p^{\overline{y}_1}+\overline{A}_2p^{\overline{y}_2}
&\mathrm{on}\;(\overline{p}_2^*,\overline{p}_1^*]\\
k_1p-r^{-1}C+\overline{B}_1p^{z_1}+\overline{B}_2p^{z_2} &\mathrm{on}\;(\overline{p}_1^*,+\infty)
\end{array}
\right.,
\]
respectively. Here
\[
\left[
\begin{array}{c}
\overline{A}_1\\
\overline{A}_2
\end{array}
\right]=
\left[
\begin{array}{cc}
{\overline{p}_2^*}^{-y_1}&0\\
0&{\overline{p}_2^*}^{-y_2}
\end{array}
\right]
\left[
\begin{array}{c}
\overline{a}_{11}\overline{p}_1^*+\overline{b}_{11}\\
\overline{a}_{21}\overline{p}_1^*+\overline{b}_{21}
\end{array}
\right],
\]
\[
\left[
\begin{array}{c}
\overline{B}_1\\
\overline{B}_2
\end{array}
\right]=
\lambda_2\left[
\begin{array}{cc}
{\overline{p}_1^*}^{-z_1}&0\\
0&{\overline{p}_1^*}^{-z_2}
\end{array}
\right]
\left[
\begin{array}{cc}
w_2(z_1)&w_2(z_2)\\
w_2(z_1)z_1&w_2(z_2)z_2
\end{array}
\right]^{-1}
\left[
\begin{array}{c}
-k_2\overline{p}_1^*+\frac{C-rK}{r}\\
-k_2\overline{p}_1^*
\end{array}
\right].
\]
$\mathrm{(3)}$ The equation
\[
\left[
\begin{array}{c}
\widetilde{a}_{11}p+\widetilde{b}_{11}\\
\widetilde{a}_{21}p+\widetilde{b}_{21}
\end{array}
\right]=
\left[
\begin{array}{c}
\widetilde{a}_{12}p+\widetilde{b}_{12}\\
\widetilde{a}_{22}p+\widetilde{b}_{22}
\end{array}
\right]
\]
has a positive solution $p^*$, where
\[
\left[
\begin{array}{c}
\widetilde{a}_{11}\\
\widetilde{a}_{12}
\end{array}
\right]=
\left[
\begin{array}{cc}
1&1\\
z_1&z_2
\end{array}
\right]^{-1}
\left[
\begin{array}{c}
-k_2\\
-k_2
\end{array}
\right],\;
\left[
\begin{array}{c}
\widetilde{b}_{11}\\
\widetilde{b}_{12}
\end{array}
\right]=
\left[
\begin{array}{cc}
1&1\\
z_1&z_2
\end{array}
\right]^{-1}
\left[
\begin{array}{c}
\frac{C-rK}{r}\\
0
\end{array}
\right],
\]
\[
\left[
\begin{array}{c}
\widetilde{a}_{12}\\
\widetilde{a}_{22}
\end{array}
\right]=
\lambda_1\left[
\begin{array}{cc}
w_1(z_1)&w_1(z_2)\\
w_1(z_1)z_1&w_1(z_2)z_2
\end{array}
\right]^{-1}
\left[
\begin{array}{c}
-k_1\\
-k_1
\end{array}
\right],
\]
and
\[
\left[
\begin{array}{c}
\widetilde{b}_{12}\\
\widetilde{b}_{22}
\end{array}
\right]=
\lambda_1\left[
\begin{array}{cc}
w_1(z_1)&w_1(z_2)\\
w_1(z_1)z_1&w_1(z_2)z_2
\end{array}
\right]^{-1}
\left[
\begin{array}{c}
\frac{C-rK}{r}\\
0
\end{array}
\right].
\]
In addition, the optimal stopping time $\tau^*$ is given by
\[\tau^*=\inf\{t:t>0, P(t)\in(0,p^*]\},
\]
and the $J_i$'s are given by
\[
J_1(p)=\left\{
\begin{array}{ll}
-K &\mathrm{on}\;(0,p^*]\\
k_2p-r^{-1}C+\widetilde{B}_1p^{z_1}+\widetilde{B}_2p^{z_2} &\mathrm{on}\;(p^*,+\infty)
\end{array}
\right.
\]
and
\[
J_2(p)=\left\{
\begin{array}{ll}
-K &\mathrm{on}\;(0,p^*]\\
k_1p-r^{-1}C+\lambda_1^{-1}w_1(z_1)\widetilde{B}_1p^{z_1}+\lambda_1^{-1}w_1(z_2)\widetilde{B}_2p^{z_2} &\mathrm{on}\;(p^*,+\infty),
\end{array}
\right.
\]
respectively. Here
\[
\left[
\begin{array}{c}
\widetilde{B}_1\\
\widetilde{B}_2
\end{array}
\right]=
\left[
\begin{array}{cc}
{p^*}^{-z_1}&0\\
0&{p^*}^{-z_2}
\end{array}
\right]
\left[
\begin{array}{c}
\widetilde{a}_{11}p^*+\widetilde{b}_{11}\\
\widetilde{a}_{21}p^*+\widetilde{b}_{21}
\end{array}
\right].
\]
\end{thm}

\begin{proof}
1.  Since $C>rK$, it follows from Lemma \ref{l:ViscZero} that $J_i(0^+)=-K$ and $\mathcal{S}_i\neq\emptyset$. In addition, by Lemma \ref{t:OptiStra}, we
have $\mathcal{S}_i\subset(0,C-rK]$. Therefore, by Theorem \ref{t:OptiStraForm} and Lemma \ref{l:Conn}, it follows that $\mathcal{S}_i=(0,p_i^*]$ for some
positive number $p_i^*$.

\noindent 2. \texttt{Case 1} $p_1^*<p_2^*$. In this case, by Theorem \ref{t:ViscSpec}, we have
\begin{equation*}
\left\{
\begin{array}{ll}
J_1(p)=-K\\
J_2(p)=-K
\end{array}
\right.
\end{equation*}
on $(0,p_1^*]$,
\begin{equation}\label{e:J1J2p1p2}
\left\{
        \begin{array}{ll}
(r+\lambda_1)J_1(p)-\mu_1pJ_1'(p)-\frac{1}{2}\sigma_1^2p^2J_1''(p)-\lambda_1J_2(p)-p+C=0\\
J_2(p)=-K
\end{array}
\right.
\end{equation}
on $(p_1^*,p_2^*]$,
and
\begin{equation}\label{e:J1J2p2Infi}
\left\{
        \begin{array}{ll}
(r+\lambda_1)J_1(p)-\mu_1pJ_1'(p)-\frac{1}{2}\sigma_1^2p^2J_1''(p)-\lambda_1J_2(p)-p+C=0\\
(r+\lambda_2)J_2(p)-\mu_2pJ_2'(p)-\frac{1}{2}\sigma_2^2p^2J_2''(p)-\lambda_2J_1(p)-p+C=0
\end{array}
\right.
\end{equation}
on $(p_2^*,+\infty)$.

Note that $r+\lambda_1-\mu_1>0$ by Step 2 of the proof of Lemma \ref{l:Soluz}. We have by (\ref{e:J1J2p1p2})
\[
J_1(p)=\frac{p}{r+\lambda_1-\mu_1}-\frac{\lambda_1K+C}{r+\lambda_1}+A_1p^{y_1}+A_2p^{y_2}\;\,\mathrm{on}\;(p_1^*,p_2^*],
\]
where $A_1$ and $A_2$ are two constants, and $y_1$ and $y_2$ are the two solutions of the following quadratic equation
\[
r+\lambda_1-\mu_1y-\frac{1}{2}\sigma_1^2y(y-1)=0
\]
with $y_1<y_2$.

Since $J_1$ and $J_2$ are Lipschitz continuous, we have by (\ref{e:J1J2p2Infi})
\[
\left\{
\begin{array}{ll}
J_1(p)=\dfrac{(r+\lambda_1+\lambda_2-\mu_2)p}
{(r+\lambda_1-\mu_1)(r+\lambda_2-\mu_2)-\lambda_1\lambda_2}-\dfrac{C}{r}+B_1p^{z_1}+B_2p^{z_2}\\
J_2(p)=\dfrac{(r+\lambda_1+\lambda_2-\mu_1)p}
{(r+\lambda_1-\mu_1)(r+\lambda_2-\mu_2)-\lambda_1\lambda_2}-\dfrac{C}{r}+\lambda_1^{-1}w_1(z_1)B_1p^{z_1}\\
\qquad\qquad+\lambda_1^{-1}w_1(z_2)B_2p^{z_2}
\end{array}
\right.
\]
on $(p_2^*,+\infty)$, where $B_1$ and $B_2$ are some constants, and $z_1$ and $z_2$ are the solutions introduced in Lemma \ref{l:Soluz}.

Therefore, by $C^1$ continuity of $J_1$ and $J_2$, we have
\begin{equation}\label{e:Syst1}
\left\{
\begin{array}{ll}
\dfrac{p_1^*}{r+\lambda_1-\mu_1}-\dfrac{\lambda_1K+C}{r+\lambda_1}+A_1{p_1^*}^{y_1}+A_2{p_1^*}^{y_2}=-K\\
\dfrac{1}{r+\lambda_1-\mu_1}+A_1y_1{p_1^*}^{y_1-1}+A_2y_2{p_1^*}^{y_2-1}=0,
\end{array}
\right.
\end{equation}
\begin{equation}\label{e:Syst2}
\left\{
\begin{array}{llll}
\dfrac{p_2^*}{r+\lambda_1-\mu_1}-\dfrac{\lambda_1K+C}{r+\lambda_1}+A_1{p_2^*}^{y_1}+A_2{p_2^*}^{y_2}\\
\qquad=\dfrac{(r+\lambda_1+\lambda_2-\mu_2)p_2^*}{(r+\lambda_1-\mu_1)(r+\lambda_2-\mu_2)-\lambda_1\lambda_2}-\dfrac{C}{r}+B_1{p_2^*}^{z_1}+B_2{p_2^*}^{z_2}\\
\dfrac{1}{r+\lambda_1-\mu_1}+A_1y_1{p_2^*}^{y_1-1}+A_2y_2{p_2^*}^{y_2-1}\\
\qquad=\dfrac{r+\lambda_1+\lambda_2-\mu_2}{(r+\lambda_1-\mu_1)(r+\lambda_2-\mu_2)-\lambda_1\lambda_2}+B_1z_1{p_2^*}^{z_1-1}+B_2z_2{p_2^*}^{z_2-1},
\end{array}
\right.
\end{equation}
and
\begin{equation}\label{e:Syst3}
\left\{
\begin{array}{ll}
\dfrac{(r+\lambda_1+\lambda_2-\mu_1)p_2^*}{(r+\lambda_1-\mu_1)(r+\lambda_2-\mu_2)-\lambda_1\lambda_2}-\dfrac{C}{r}\\
\qquad+\lambda_1^{-1}w_1(z_1)B_1{p_2^*}^{z_1}+\lambda_1^{-1}w_1(z_2)B_2{p_2^*}^{z_2}=-K\\
\dfrac{r+\lambda_1+\lambda_2-\mu_1}{(r+\lambda_1-\mu_1)(r+\lambda_2-\mu_2)-\lambda_1\lambda_2}\\
\qquad+\lambda_1^{-1}w_1(z_1)B_1z_1{p_2^*}^{z_1-1}+\lambda_1^{-1}w_1(z_2)B_2z_2{p_2^*}^{z_2-1}=0.
\end{array}
\right.
\end{equation}

By solving $A_1$ and $A_2$ from (\ref{e:Syst1}) and solving $A_1$ and $A_2$ from (\ref{e:Syst2}) and (\ref{e:Syst3}), we get
\[
\left[
\begin{array}{cc}
{p_1^*}^{-y_1}&0\\
0&{p_1^*}^{-y_2}
\end{array}
\right]
\left[
\begin{array}{c}
a_{11}p_1^*+b_{11}\\
a_{21}p_1^*+b_{21}
\end{array}
\right]=
\left[
\begin{array}{cc}
{p_2^*}^{-y_1}&0\\
0&{p_2^*}^{-y_2}
\end{array}
\right]
\left[
\begin{array}{c}
a_{12}p_2^*+b_{12}\\
a_{22}p_2^*+b_{22}
\end{array}
\right].
\]

\noindent \texttt{Case 2} $p_1^*>p_2^*$. This case is similar to Case 1.

\noindent \texttt{Case 3} $p_1^*=p_2^*=:p^*$. By Theorem \ref{t:ViscSpec}, we have
\begin{equation*}
\left\{
\begin{array}{ll}
J_1(p)=-K\\
J_2(p)=-K
\end{array}
\right.
\end{equation*}
on $(0,p^*]$,
and
\begin{equation}\label{e:J1J2pInfi}
\left\{
        \begin{array}{ll}
(r+\lambda_1)J_1(p)-\mu_1pJ_1'(p)-\frac{1}{2}\sigma_1^2p^2J_1''(p)-\lambda_1J_2(p)-p+C=0\\
(r+\lambda_2)J_2(p)-\mu_2pJ_2'(p)-\frac{1}{2}\sigma_2^2p^2J_2''(p)-\lambda_2J_1(p)-p+C=0
\end{array}
\right.
\end{equation}
on $(p^*,+\infty)$.

Since $J_1$ and $J_2$ are Lipschitz continuous, we have by (\ref{e:J1J2pInfi})
\[
\left\{
\begin{array}{ll}
J_1(p)=\dfrac{(r+\lambda_1+\lambda_2-\mu_2)p}
{(r+\lambda_1-\mu_1)(r+\lambda_2-\mu_2)-\lambda_1\lambda_2}-\dfrac{C}{r}+\widetilde{B}_1p^{z_1}+\widetilde{B}_2p^{z_2}\\
J_2(p)=\dfrac{(r+\lambda_1+\lambda_2-\mu_1)p}
{(r+\lambda_1-\mu_1)(r+\lambda_2-\mu_2)-\lambda_1\lambda_2}-\dfrac{C}{r}\\
\qquad\qquad+\lambda_1^{-1}w_1(z_1)\widetilde{B}_1p^{z_1}+\lambda_1^{-1}w_1(z_2)\widetilde{B}_2p^{z_2}
\end{array}
\right.
\]
on $(p^*,+\infty)$, where $\widetilde{B}_1$ and $\widetilde{B}_2$ are some constants, and $z_1$ and $z_2$ are the solutions introduced in Lemma \ref{l:Soluz}.

Therefore, by $C^1$ continuity of $J_1$ and $J_2$, we have
\begin{equation}\label{e:SystEqua1}
\left\{
\begin{array}{ll}
\dfrac{(r+\lambda_1+\lambda_2-\mu_2)p^*}{(r+\lambda_1-\mu_1)(r+\lambda_2-\mu_2)-\lambda_1\lambda_2}
-\dfrac{C}{r}+\widetilde{B}_1{p^*}^{z_1}+\widetilde{B}_2{p^*}^{z_2}=-K\\
\dfrac{r+\lambda_1+\lambda_2-\mu_2}{(r+\lambda_1-\mu_1)(r+\lambda_2-\mu_2)-\lambda_1\lambda_2}+\widetilde{B}_1z_1{p^*}^{z_1-1}+\widetilde{B}_2z_2{p^*}^{z_2-1}=0,
\end{array}
\right.
\end{equation}
and
\begin{equation}\label{e:SystEqua2}
\left\{
\begin{array}{ll}
\dfrac{(r+\lambda_1+\lambda_2-\mu_1)p^*}{(r+\lambda_1-\mu_1)(r+\lambda_2-\mu_2)-\lambda_1\lambda_2}-\dfrac{C}{r}\\
\qquad\qquad+\lambda_1^{-1}w_1(z_1)\widetilde{B}_1{p^*}^{z_1}+\lambda_1^{-1}w_1(z_2)\widetilde{B}_2{p^*}^{z_2}=-K\\
\dfrac{r+\lambda_1+\lambda_2-\mu_1}{(r+\lambda_1-\mu_1)(r+\lambda_2-\mu_2)-\lambda_1\lambda_2}\\
\qquad\qquad+\lambda_1^{-1}w_1(z_1)\widetilde{B}_1z_1{p^*}^{z_1-1}+\lambda_1^{-1}w_1(z_2)\widetilde{B}_2z_2{p^*}^{z_2-1}=0.
\end{array}
\right.
\end{equation}

By solving $\widetilde{B}_1$ and $\widetilde{B}_2$ from (\ref{e:SystEqua1}) and solving $\widetilde{B}_1$ and $\widetilde{B}_2$ from (\ref{e:SystEqua2}), we get
\[
\left[
\begin{array}{c}
\widetilde{a}_{11}p^*+\widetilde{b}_{11}\\
\widetilde{a}_{21}p^*+\widetilde{b}_{21}
\end{array}
\right]=
\left[
\begin{array}{c}
\widetilde{a}_{12}p^*+\widetilde{b}_{12}\\
\widetilde{a}_{22}p^*+\widetilde{b}_{22}
\end{array}
\right].
\]
The proof is complete.
\end{proof}

\begin{cor}
If $r>\mu_2$ and $C>rK$, then the conclusions in Theorem \ref{t:optimal3} hold.
\end{cor}
\begin{proof}
The proof can be completed by the fact that $r>\mu_2$ implies $r>x_2$.
\end{proof}

\begin{exm}
Take $\mu_1=0.01$, $\mu_2=0.10$, $\sigma_1=\sigma_2=0.25$, $r=0.08$, $\lambda_1=\lambda_2=0.05$, $C=20$, $K=5$. Then we have $r=0.08>0.0723=x_2$ and $C=20>0.40=rK$. Thus we apply Theorem \ref{t:optimal3}, and find that (2) of Theorem \ref{t:optimal3} gives us $(\overline{p}_1^*,\overline{p}_2^*)=(2.08,1.04)$ and the optimal stopping time $\tau^*=\inf\{t:t>0, (X(t),P(t))\in\{e_1\}\times(0,2.08]\cup\{e_2\}\times(0,1.04]\}$.

It is interesting to compare regime-switching cases with no regime-switching cases. If there is no regime switching and the price $P$ satisfies
\[
\mathrm{d}P(t)=0.01P(t)\mathrm{d}t+0.25P(t)\mathrm{d}B(t),
\]
the optimal stopping time is $\inf\{t:t>0, P(t)\in(0,12.73]\}$. However, if the price $P$ satisfies
\[
\mathrm{d}P(t)=0.10P(t)\mathrm{d}t+0.25P(t)\mathrm{d}B(t),
\]
the firm should never stop the extraction since $r=0.08<0.10$.

In summary, the firm may stop the extraction even though it should never stop the extraction in one of regimes.
\end{exm}

\section{Conclusions}\label{Conclu}

Regime-switching processes are introduced to describe the price of financial assets and commodities \cite[e.g.,][]{CherGallGhysTauc2003, CasaDufrRout2005}, the stochastic behavior of temperature \cite{ElWaFa2014}, and so on. Under the assumption that underlying processes are modeled by some regime-switching processes, Guo and Zhang \cite{GuoZhang2004} derive an explicit closed solution for perpetual American options, Bae \textit{et al.}~\cite{BaKiMu2014} investigate dynamic asset allocation among diverse financial markets, and Elias \textit{et al.}~\cite{ElWaFa2014} valuate temperature-based weather options. These motivate us to study in a uniform way the optimal stopping problems in which underlying processes and payoff functions are modulated by Markov chains.

In this paper, we employ the viscosity solution technique to analyze optimal stopping problems with regime switching. Specifically, we first prove the value function is a viscosity solution of some variational inequalities; next, we obtain the uniqueness of the viscosity solution of the variational inequalities; then, we study the regularity of the value function and the form of optimal stopping intervals.

In Section \ref{model}, we provide an application of the results obtained in Section \ref{optstop}. At the end of the paper, a numerical example is demonstrated. From the example, we come to a conclusion that a firm may stop a project even though it should never stop the project in one of regimes.

\bigskip
\bigskip
\bigskip

\bibliographystyle{mybst2}
\bibliography{mybibfile}

\providecommand{\bysame}{\leavevmode\hbox to3em{\hrulefill}\thinspace}
\providecommand{\MR}{\relax\ifhmode\unskip\space\fi MR }
\providecommand{\MRhref}[2]{%
  \href{http://www.ams.org/mathscinet-getitem?mr=#1}{#2}
}
\providecommand{\href}[2]{#2}
\begin{thebibliography}{10}

\bibitem{BaKiMu2014}
{ G.~I. Bae, W.~C. Kim, and J.~M. Mulvey}, \emph{Dynamic asset allocation for
  varied financial markets under regime switching framework}, European Journal
  of Operational Research \textbf{234} (2014), no.~2, 450--458.

\bibitem{Bollen1999}
{ N.~P.~B. Bollen}, \emph{Real options and product life cycles}, Management
  Science \textbf{45} (1999), no.~5, 670--684.

\bibitem{BoyaLeve2008}
{ S.~Boyarchenko and S.~Levendorski\v{i}}, \emph{Exit problems in
  regime-switching models}, Journal of Mathematical Economics \textbf{44}
  (2008), no.~2, 180--206.

\bibitem{BuffElli2002a}
{ J.~Buffington and R.~J. Elliott}, \emph{American options with regime
  switching}, International Journal of Theoretical and Applied Finance
  \textbf{5} (2002), no.~5, 497--514.

\bibitem{BuffElli2002}
\leavevmode\vrule height 2pt depth -1.6pt width 23pt, \emph{Regime switching
  and {European} options}, Stochastic Theory and Control \textbf{280} (2002),
  73--84.

\bibitem{CasaDufrRout2005}
{ J.~Casassus, P.~Collin-Dufresne, and B.~R. Routledge}, \emph{Equilibrium
  commodity prices with irreversible investment and non-linear technology},
  {NBER} working paper No. 11864, 2005.

\bibitem{CherGallGhysTauc2003}
{ M.~Chernova, A.~R. Gallantb, E.~Ghyselsb, and G.~Tauchen}, \emph{Alternative
  models for stock price dynamics}, Journal of Econometrics \textbf{116}
  (2003), no.~1--2, 225--257.

\bibitem{CIL92}
{ M.~G. Crandall, H.~Ishii, and P.-L. Lions}, \emph{User’s guide to viscosity
  solutions of second order partial differential equations}, Bulletin of the
  American Mathematical Society \textbf{27} (1992), no.~1, 1--67.

\bibitem{d2012markov}
{ B.~D'Auria and O.~Kella}, \emph{Markov modulation of a two-sided reflected
  Brownian motion with application to fluid queues}, Stochastic Processes and
  their Applications \textbf{122} (2012), no.~4, 1566--1581.

\bibitem{el1982optimal}
{ N.~El~Karoui, J.-P. Lepeltier, and B.~Marchal}, \emph{Optimal stopping of
  controlled Markov processes}, Advances in Filtering and Optimal Stochastic
  Control, Springer-Verlag, Berlin, 1982, pp.~106--112.

\bibitem{ElWaFa2014}
{ R.~S. Elias, M.~I.~M. Wahab, and L.~Fang}, \emph{A comparison of
  regime-switching temperature modeling approaches for applications in weather
  derivatives}, European Journal of Operational Research \textbf{232} (2014),
  no.~3, 549--560.

\bibitem{EloeLiuYinZhang2008}
{ P.~Eloe, R.~H. Liu, M.~Yatsuki, G.~Yin, and Q.~Zhang}, \emph{Optimal selling
  rules in a regime-switching exponential Gaussian diffusion model}, SIAM
  Journal on Applied Mathematics \textbf{69} (2008), no.~3, 810--829.

\bibitem{GuoZhang2004}
{ X.~Guo and Q.~Zhang}, \emph{Closed-form solutions for perpetual {American}
  put options with regime switching}, SIAM Journal on Applied Mathematics
  \textbf{64} (2004), no.~6, 2034--2049.

\bibitem{GuoZhang2005}
\leavevmode\vrule height 2pt depth -1.6pt width 23pt, \emph{Optimal selling
  rules in a regime switching model}, IEEE Transactions on Automatic Control
  \textbf{50} (2005), no.~9, 1450--1455.

\bibitem{Guo2001}
{ X.~Guo}, \emph{An explicit solution to an optimal stopping problem with
  regime switching}, Journal of Applied Probability \textbf{38} (2001), no.~2,
  464--481.

\bibitem{Kle98}
{ F.~C. Klebaner}, \emph{Introduction to {Stochastic} {Calculus} with
  {Application}}, Imperrial College Press, London, 1998.

\bibitem{Pham2009}
{ H.~Pham}, \emph{Continuous-time {Stochastic} {Control} and {Optimization}
  with {Financial} {Applications}}, Spriger-Verlag, Berlin, 2009.

\bibitem{Touzi2013}
{ N.~Touzi}, \emph{Optimal Stochastic Control, Stochastic Target Problems, and
  Backward SDE}, Spriger-Verlag, New York, 2013.

\bibitem{Zhang1996}
{ Z.~Zhang}, \emph{Lecture on {Differential} {Topology} \emph{(in {Chinese})}},
  Peking University Press, Beijing, 1996.

\end{thebibliography}

\end{document}